\documentclass[final,12pt]{amsart}
\usepackage{amssymb, amsmath, amsthm,geometry}
\usepackage{mathrsfs}

\usepackage{graphicx}
\usepackage[colorlinks=true, citecolor=blue, linkcolor=blue, urlcolor=blue]{hyperref}
\newgeometry{asymmetric, centering}
 \usepackage{xcolor}

\usepackage{ifdraft}
\ifoptionfinal{
\usepackage[disable]{todonotes}
}{
\usepackage[norefs, nocites]{refcheck}
\usepackage{graphicx}
\usepackage{epstopdf}
\usepackage{tikz}
\usepackage{tikz-3dplot}
\usetikzlibrary{3dtools,calc}

\usepackage[notref, notcite]{showkeys}
\usepackage[bordercolor=white, color=white]{todonotes}
}

\makeatletter\providecommand\@dotsep{5}\def\listtodoname{List of Todos}\def\listoftodos{\hypersetup{linkcolor=black}\@starttoc{tdo}\listtodoname\hypersetup{linkcolor=blue}}\makeatother

\allowdisplaybreaks[4]

\numberwithin{equation}{section}

\newtheorem{lemma}{Lemma}[section]
\newtheorem{proposition}{Proposition}[section]
\newtheorem{theorem}{Theorem}[section]

\newtheorem{definition}{Definition}[section]

\theoremstyle{remark}
\newtheorem*{claim}{Claim}

\newcommand{\bel}{\begin{equation} \label}
\newcommand{\ee}{\end{equation}}
\def\beq{\begin{equation}}
\def\eeq{\end{equation}}
\newcommand{\bea}{\begin{eqnarray}}
\newcommand{\eea}{\end{eqnarray}}
\newcommand{\beas}{\begin{eqnarray*}}
\newcommand{\eeas}{\end{eqnarray*}}

\renewcommand{\div}{\mathrm{div}\,}  

\def\C{\mathbb C}
\def\E{\mathcal E}

\def\R{\mathbb R}

\def\Z{\mathbb Z}

\def\Char{\mathrm{Char}}
\def\WF{\mathrm{WF}}
\def\singsupp{\mathrm{singsupp}}
\def\out{\mathrm{out}}

\def\sub{\mathrm{sub}}

\def\sgn{\mathrm{sgn}}
\def\b{\backslash}
\def \into{\mathrm{in}}

\def \sub{\mathrm{sub}}

\def \id{\mathrm{id}}

\def\E{\mathcal E}

\def\W{\mathcal W}

\def\D{\mathbb D}

\def \e{\varepsilon}

\def\l{\left}
\def\r{\right}

\newcommand\rotom{\mho}

\DeclareMathOperator{\supp}{supp}

\date{Compiled \today}
\title
[Inverse connection problem]{Inverse problem for connections in semi-linear wave equations on Lorentzian manifolds}

\author{Lauri Oksanen}
\address{Department of Mathematics and Statistics, University of Helsinki, Helsinki, Finland}
\email{lauri.oksanen@helsinki.fi}

\author{Ruochong Zhang}
\address{Department of Mathematics, Fudan University, 
Handan Road, Shanghai China, No.220.}
\email{22110180055@m.fudan.edu.cn}

\begin{document}

\begin{abstract} This paper recovers Hermitian connections of semi-linear wave equations with cubic nonlinearity. The main novelty is in the geometric generality: we treat the case of an arbitrary globally hyperbolic Lorentzian manifold. Our approach is based on  microlocal analysis of nonlinear wave interactions, which recovers a non-abelian broken light-ray transform, and the inversion of broken light-ray transforms on globally hyperbolic Lorentzian manifolds. 
\end{abstract}

\maketitle

\tableofcontents

\section{Introduction}
Let $(M,g)$ be a time-oriented Lorentzian manifold of dimension $1+n$ for $n\ge 2$ and signature $(-, +,\dots,+)$. Consider the trivial Hermitian vector bundle $E = M\times \C^{m}$, equipped with a bundle metric $\langle\cdot,\cdot\rangle_{E}$, given by the standard Hermitian inner product on $\C^{m}$. Let 
$$
\nabla: C^{\infty}(M;E) \to C^{\infty}(M;T^{\ast}M \otimes E)
$$
be a connection on $E$ that is compatible with the bundle metric. On the trivial Hermitian vector bundle, such a connection can be written as
$$
  \nabla = d+A,
$$
where $d$ denotes the exterior derivative and $A$ is a skew-Hermitian matrix-valued $1-$form. Define the connection wave operator by 
\begin{equation}\label{eqn: connection wave operator}
\Box_{g,A} =(d+A)^{\ast}(d+A),
\end{equation}
where $(d+A)^{\ast}$ denotes the adjoint of $d+A$ with respect to the $L^{2}$ inner product on $L^{2}(M;T^{\ast}M\otimes E)$. We then consider the following semi-linear wave equation with a source:
\begin{equation}\label{eqn: connection wave equation with a source}
   \Box_{g,A} u + |u|^{2}u=f,
\end{equation}
where $u$ is a section of $E$ and $|\cdot|$ denotes the norm induced by the bundle metric.

A typical assumption to ensure that the wave equation \eqref{eqn: connection wave equation with a source} is well posed is that $(M,g)$ is globally hyperbolic. Such Lorentzian manifolds are characterized via causal structures, as described in \cite{Beem-Ehrlich-Easley,O}. For $x,y\in M$, the notation $x<y$ indicates the existence of a future-pointing causal curve from $x$ to $y$, while $x\ll y$ indicates the existence of a future-pointing timelike curve from $x$ to $y$. We write $x\le y$ if $x=y$ or $x<y$.

Define the causal future and past of a point $x\in M$ respectively by 
$$
  J^{+}(x) = \{y\in M; x\le y\},\quad\text{and}\quad J^{-}(x) = \{y\in M; y\le x\}.
$$
Similarly, define the chronological future and chronological past by
$$
  I^{+}(x) =\{y\in M; x\ll y\},\quad\text{and}\quad I^{-}(x) = \{y\in M; y\ll x\}.
$$
Given a set $A\subseteq M$, define the causal and chronological futures and pasts of $A$ by
$$
  J^{\pm}(A) = \bigcup_{a\in A}J^{\pm}(a)\quad\text{and}\quad I^{\pm}(A) = \bigcup_{a\in A}I^{\pm}(a).
$$

It was shown in \cite{Bernal-Sanchez-2007} that a time-oriented Lorentzian manifold $(M,g)$ is \textit{globally hyperbolic} if and only if $M$ is noncompact and the causal diamond $J^{+}(p)\cap J^{-}(q)$ is compact in $M$ for all $p,q\in M$.

 Under the assumption of globally hyperbolicity, $(M,g)$ is isometric to $\R\times N$.  Let $(x^{0},x^{1},x^{2},x^{3})$ be the local coordinates on $(M,g)$, where $t:=x^0$ represents time and $x^{\prime}:=(x^{1},x^{2},x^{3})$ represents spatial coordinates. The metric admits a decomposition into time and space
$$
   g = -\beta(t,x^{\prime})dt^{2} + g_0(t,x^{\prime}),
$$
where $\beta$ is a smooth positive function and $g_0(t,\cdot)$ is a family of Riemannian metrics on $N$, smoothly depending on $t$. For each fixed $t$, the slice $\{t\}\times N$ is a Cauchy hypersurface. The Cauchy problem for the wave equation, with initial data given on a Cauchy hypersurface, is well-posed, see \cite[Theorem 12.19]{Ringstrom}.

\subsection{Forward problem}
Consider the Cauchy problem for a semi-linear wave equation on a globally hyperbolic manifold $(M,g)$, 
\begin{equation}\label{eqn: semi-linear wave equation on a globally hyperbolic manifold}
    \begin{aligned}
       \begin{cases}
       \Box_{g,A} u(x) + |u(x)|^{2}u(x) = f(x),\quad &x\in (0,+\infty)\times N,\\
        u(0,x^{\prime} ) = 0, \partial_{t}(0,x^{\prime}) = 0,\quad &x^{\prime} \in N.  
       \end{cases}
    \end{aligned}
\end{equation}

Denote by $B(0,d)$ the open geodesic ball on the Riemannian manifold $(\{0\}\times N, g_{0}(0,\cdot))$, which is geodesically convex for sufficiently small $\rho>0$.
\begin{equation}\label{eqn: mou}
   \rotom = (0,T)\times B(0,\rho).
\end{equation}
The local well-posedness of semi-linear wave equations is well-established, see, for example, \cite[Theorem I-III]{Hughes-Kato-Marsden-1976}, \cite[Theorem 6]{K} and \cite[Theorem 6.3.1]{Rauch-book}. Let $r_{\rotom}>0$ be sufficiently small. We define the space of admissible sources as
$$
  \mathscr{C}_{\rotom} = \l\{f\in H^{k}(M;E); k\ge 4, \supp f \subseteq \rotom , \|f\|_{H^{k}(M)} \le r_{\rotom}\r\}.
$$
Given $f\in \mathscr{C}_{\rotom}$, the Cauchy problem \eqref{eqn: semi-linear wave equation on a globally hyperbolic manifold} admits a unique solution
$$
  u \in H^{k+1}((0,T)\times N; E).
$$
Consequently, the source-to-solution map, which models measurements, is well-defined.
\begin{equation}\label{eqn: source-to-solution map}
   L_{g,A} :  \mathscr{C}_{\rotom} \longrightarrow H^{k+1}(\rotom;E),\quad L_{g,A}(f) = u|_{\rotom},\quad\text{for all }f\in \mathscr{C}_{\rotom}.
\end{equation}

\subsection{Main theorem}
Given the metric $g$ and the source-to-solution map $L_{g,A}$, the inverse problem is to recover the connection $A$ on the causal diamond
\begin{equation}\label{eqn: causal diamond}
  \D = J^{+}(\rotom) \cap J^{-}(\rotom), 
\end{equation}
which is the maximal region where $A$ can be determined due to the finite speed of propagation. The connection $A$ can be recovered only up to a natural gauge transformation
\begin{equation}\label{eqn: gauge transformation}
   A \triangleleft \boldsymbol{\varphi}  = \boldsymbol{\varphi}^{-1}d\boldsymbol{\varphi} + \boldsymbol{\varphi}^{-1}A\boldsymbol{\varphi},
\end{equation}
where
\begin{equation}\label{eqn: natural gauge}
\boldsymbol{\varphi} \in C^{\infty}(\D;U(n))\quad\text{and}\quad \boldsymbol{\varphi}|_{\rotom} = \id.
\end{equation}

Furthermore, since $\rotom$ is the obeservation region, it is natural to assume that $A$ is known a priori in $\rotom$.

\begin{theorem}\label{thm: main theorem}
   Let $A$ and $B$ be two Hermitian connections such that $L_{g,A} = L_{g,B}$ and $A=B$ on $\rotom$. Then, there exists a gauge transformation $\boldsymbol{\varphi}$ satisfying \eqref{eqn: natural gauge} such that
   $$
  A = B \triangleleft \boldsymbol{\varphi}.
   $$
\end{theorem}

\subsection{Previous literature}
Inverse problems for linear wave equations have been extensively studied. The Boundary Control (BC) method, developed by Belishev \cite{Belishev}, see also \cite{BK}, combined with the unique continuation results of Tataru \cite{T}, provides a powerful framework for recovering time-independent coefficients in general hyperbolic equations. Kurylev, Oksanen and Paternain \cite{Kurylev-Oksanen-Paternain-JDG-2018} recovered time-independent Hermitian connections by BC-method. However, this approach fails in the case of time-dependent coefficients since unique continuation does not hold, as shown by the counterexample of Alinhac \cite{A}.

A breakthrough in the case of nonlinear wave equations was achieved by Kurylev, Lassas, and Uhlmann \cite{KLU}, who introduced a method based on higher-order linearization and microlocal analysis of nonlinear wave interactions, enabling the recovery of time-dependent coefficients in nonlinear hyperbolic equations. Nevertheless, the reconstruction of time-dependent coefficients for linear wave equations remains an open problem. There are some important results on such problem under specific geometric assumptions. We refer the readers to \cite{Alexakis-Feizmohammadi-Oksanen-2022, Alexakis-Feizmohammadi-Oksanen-2024,Oksanen-Rakesh-Salo-2024,Stefanov-Yang-APDE,Stefanov-1989} for these results.

In contrast to leading-order coefficients in linear terms and the nonlinearities, lower-order coefficients---such as connections and potentials---are more difficult to observe, as they have weaker effect on the associated soluions. Chen, Lassas, Oksanen, and Paternain \cite{CLOP,CLOP-CMP,CLOP-Annalen} utilized broken light-ray transforms to reconstruct connections. Feizmohammadi and Oksanen \cite{FO} used Gaussian beams and introduced  truncated integral transforms to recover potentials. Chen, Lu, and Zhang \cite{Chen-Lu-Zhang-2025} introduced the method of lower-order symbol calculus to recover potentials.

We refer to \cite{FLO,Hintz-Uhlmann-Zhai-CPDE,Hintz-Uhlmann-Zhai-IMRN,Kian,LUW,SaBarreto-Stefanov,SaBarreto-Uhlmann-Wang,WZ} for recent advances in reconstruction of time-dependent coefficients for nonlinear wave equations, and to \cite{KLOU,Uhlmann-Wang-CPAM,Uhlmann-Zhai-JMPA,Uhlmann-Zhai-Annalen,Uhlmann-Zhang} for physical applications.

Here, we recover Hermitian connections on a globally hyperbolic Lorentzian manifold from the source-to-solution map and the underlying Lorentzian metric, thereby generalizing the result of \cite{CLOP}, which addresses the Minkowski space case. The main challenges in our setting stem from the presence of cut points and conjugate points in a general Lorentzian geometry, which causes difficulties in the microlocal analysis of wave interactions, as well as from the need to invert the broken light-ray transforms in curved spacetimes. While the inversion of the broken non-Abelian X-ray transforms has recently been achieved on a Riemannian manifold \cite{St-Amant-2024}, the corresponding problem for the broken light-ray transform on a Lorentzian manifold is solved here.

This paper is organized as follows. Section \ref{sec: broken light-ray} addresses the inversion of the broken light-ray transform on a globally hyperbolic Lorentzian manifold $(M,g)$. Section \ref{sec: linear waves} formulates the connection wave equation in local coordinates and reviews the necessary background on Lagrangian distributions and intersecting pairs of Lagrangian distributions, which are used to analyze the propagation of singularities for a  linear wave equation with a connection. Section \ref{sec: nonlinear wave interactions}  reduces the recovery problem from the source-to-solution map to the inversion of the broken light-ray transform, thereby completing the proof of the main theorem.

\section{Broken light-ray transform}\label{sec: broken light-ray}
\subsection{Lorentzian geometry}
In this section, we introduce some notation and collect some results on globally hyperbolic Lorentzian manifolds $(M,g)$ that will be used in the proof of Theorem \ref{thm: main theorem}. The general references are \cite{Beem-Ehrlich-Easley,Minguzzi,O}.

 Fix a point $p\in M$. Let $L_{p}^{+}M \subseteq T_{p}M$ denote the set of future-pointing lightlike (null) tangent vectors at $p$, and let $L_{p}^{\ast,+}M \subseteq T_{p}^{\ast}M$ denote the set of future-pointing lightlike covectors. Define the bundles of future-pointing null vectors and covectors by 
 $$
 L^{+}M := \bigcup_{p\in M}L^{+}_{p}M,\quad L^{\ast,+}M := \bigcup_{p\in M}L_{p}^{\ast,+}M.
 $$ 
 Similarly, define the bundles of past-pointing lightlike vectors and covectors $L^{-}M$ and $L^{\ast,-}M$, respectively. The null tangent bundle is given by
 $$
   LM := L^{\ast}M \cup L^{-}M,
 $$
 and the null cotangent bundle is deined by 
 $$
   L^{\ast}M: = L^{\ast,+}M \cup L^{\ast,-}M.
  $$

Let $\alpha: [a,b]\to M$ be a future-pointing, piecewise smooth causal curve. The length of $\alpha$ is defined by
$$
L(\alpha) = \int_{a}^{b}\sqrt{-g(\dot{\alpha}(s),\dot{\alpha}(s))}ds.
$$
Given $x\le y$, the time separation function at $(x,y)\in M\times M$ is defined as
$$
  \tau(x,y) = \sup\{L(\alpha): \alpha \text{ is a future-pointing causal curve from $x$ to $y$} \}.
$$
According to \cite[Lemma 14.21]{O}, the function $\tau(x,y)$ is continuous on $M\times M$, provided that $(M,g)$ is globally hyperbolic. Furthermore, by  \cite[Lemma 14.19]{O}, if $x<y$, then there exists a future-pointing causal geodesic $\gamma$ from $x$ to $y$ such that 
$$
  L(\gamma) = \tau(x,y).
$$

Given $(x,v)\in L^{+}M$, the lightlike geodesic with initial value $(x,v)$ is denoted by $\gamma_{x,v}(s)$, and is defined by
$$
  \gamma_{x,v}(0) = x,\quad\text{and}\quad \dot{\gamma}_{x,v}(0) = v.
$$

It is important to note that a globally hyperbolic Lorentzian manifold is generally not geodesically complete.  Therefore, we define the maximal time of existence of the geodesic $\gamma_{x,v}$ as 
\begin{equation}\label{eqn: maximal interval}
  \mathcal{T}(x,v) = \sup\{r : \text{the geodesic } \gamma_{x,v}: [0,r) \text{ is well-defined} \}.
\end{equation}
Additionally, define the null cut function \cite[Definition 9.32]{Beem-Ehrlich-Easley} as
$$
  \rho(x,v) = \sup\{s\in [0,\mathcal{T}(x,v)): \tau(x,\gamma_{x,v}(s)) = 0\}.
$$
If $\rho(x,v)< \mathcal{T}(x,v)$, the point $p(x,v) := \gamma_{x,v}(\rho(x,v))$ exists. It is the first cut point along the geodesic $\gamma_{x,v}$.

The following Lemma \cite[Lemma 6.5]{FLO} establishes the uniquness of causal curves before the first cut point.
\begin{lemma}\label{lemma: only causal curve}
 Let $(x,v)\in L^{+}M$ and $s_0<\rho(x,v)$. Then the geodesic segment $\gamma_{x,v}([0,s_0])$ is the unique causal curve from $x$ to $\gamma_{x,v}(s_0)$.
\end{lemma}

Moreover, the null cut function $\rho$ is lower semi-continuous on a globally hyperbolic Lorentzian manifold, as shown in \cite[Theorem 9.33]{Beem-Ehrlich-Easley}.  

The following lemma \cite[Lemma 6.8]{FLO} shows that the null cut function is symmetric and hence can be naturally extended to $L^{-}M$.
\begin{lemma}\label{lemma: symmetry of the null cut function}
 Suppose $(x,v)\in L^{+}M$ and let $(y,w) = (\gamma_{x,v}(\rho(x,v)),\dot{\gamma}_{x,\xi}(\rho(x,v)))$, assuming this point is well-defined. Then,
 $$
    \rho(y,-w) = \rho(x,v).
 $$
\end{lemma}
In addition, both Lemma \ref{lemma: only causal curve} and the lower semi-continuity of $\rho(x,v)$ remain valid for $(x,v)\in L^{-}M$, see \cite[Section 6.2]{FLO}.

\subsection{Parallel transport}
Let $\gamma_{x,v}:[0,s_0]\to M$ be a lightlike geodesic such that $(x,v)\in L^{\pm}M$ and $0< s_0 < \mathcal{T}(x,v)$. Let $\langle\cdot,\cdot\rangle$ denote the natural pairing between covectors and vectors, and let $I_{n}$ denote the $n\times n$ identity matrix. Define $U^{A}: [0,s_0] \to U(n)$ to be the solution of the transport equation 
\begin{equation}\label{eqn: equation of parallel transport}
    \begin{aligned}
      \begin{cases}
           \partial_{s}U^{A}(s)+ \langle A(\gamma(s)),\dot\gamma(s)\rangle U^{A}(s) = 0, \quad s\in(0,s_0)\\
           U^{A}(0) = I_{n},
      \end{cases}
    \end{aligned}
\end{equation}
where, for notational convenience, $\gamma(s) := \gamma_{x,v}(s)$. The parallel transport map associated with $A$ along $\gamma$ from $\gamma(0)$ to $\gamma(s_0)$ is defined by
\begin{equation}\label{eqn: parallel transport map}
   \mathbf{P}^{A}_{\gamma([0,s_0])} = U^{A}(s_0).
\end{equation}
The solution $U^{A}(s^{\prime}_0)$ is also well-defined for $s^{\prime}_0<0$ as long as the geodesic $\gamma:[s_{0}^{\prime},0]\to M$ is well-defined. In this case, the parallel transport map along $\gamma$ from $\gamma(0)$ to $\gamma(s^{\prime}_0)$ by
$$
  \mathbf{P}^{A}_{\gamma_{x,v}([s_{0}^{\prime},0])} := U^{A}(s_{0}^{\prime}).
$$
A change of variables in the transport equation \eqref{eqn: equation of parallel transport} yields the identity
\begin{equation}\label{eqn: minus parameter parallel transport}
  \mathbf{P}^{A}_{\gamma_{y,-v}([-s_0,0])} = \mathbf{P}^{A}_{\gamma_{y,v}([0,s_0])},\quad s_0\in (0,\rho(x,v)).
\end{equation}

Moreover, the inverse matrix $U^{A}(s)^{-1}: [0,s_0] \to U(n)$ satisfies the following transport equation
   \begin{equation}\label{eqn: parallel transport is invertible}
    \begin{aligned}
      \begin{cases}
           \partial_{s}\l(U^{A}(s)\r)^{-1} - U^{A}(s)^{-1} \langle A(\gamma(s)),\dot\gamma(s)\rangle  = 0,\quad s\in (0,s_0)\\
           U^{A}(0)^{-1} = I_{n}.
      \end{cases}
    \end{aligned}
\end{equation}

Generally,  we define the parallel transport along $\gamma(s)$ from $a$ to $b$ for $a<b$, provided that $\gamma(s)$ is well defined on a neighborhood of $[a,b]$ as 
\begin{equation}\label{eqn: parallel transport map in a general parameter}
\mathbf{P}^{A}_{\gamma([a,b])} = \mathbf{P}^{A}_{\tilde{\gamma}([0,b-a])},\quad\text{where}\quad \tilde{\gamma}(s) = \gamma(s+a). 
\end{equation}
Under a change of variable of the parallel transport equation \eqref{eqn: equation of parallel transport}, $\mathbf{P}^{A}_{\gamma([a,b])}$ is the solution for
\begin{equation}\label{eqn: equation for parallel transport map in a general parameter}
  \begin{aligned}
    \begin{cases}
\partial_{s} U^{A}(s;a)+ \langle A(\gamma(s)),\dot\gamma(s)\rangle U^{A}(s;a) = 0, \quad s\in(a,b)\\
           U^{A}(a;a) = I_{n},
    \end{cases}
  \end{aligned}
\end{equation}
Then, the parallel transport map satisfies the following group property.
 \begin{equation}\label{eqn: patching property of the parallel transport map}
   \mathbf{P}^{A}_{\gamma([b,c])}\mathbf{P}^{A}_{\gamma([a,b])} = \mathbf{P}^{A}_{\gamma([a,c])}.
 \end{equation}

The composition of two parallel transport maps along two lightlike geodesics that intersect at exact one point is referred to as the broken light-ray transform. We define this transform specifically in the form needed for the proof of Theorem \ref{thm: main theorem}. We illustrate the parallel transport through the following picture.

\begin{figure}[htbp]
	\centering
	\begin{minipage}{0.45\textwidth}
		\centering
		\begin{tikzpicture}[scale=1]
  \draw[blue] (0,0) ellipse (1 and 0.3);
  \draw[blue] (0,-5) ellipse (1 and 0.3);
  \draw[blue] (-1,0) -- (-1,-5);
  \draw[blue] (1,0) -- (1,-5);

  \draw (1,0) -- (3.5,-2.5);
  \draw (1,-5) -- (3.5,-2.5);
  \draw (-1,0) -- (-3.5,-2.5);
  \draw (-1,-5) -- (-3.5,-2.5);

  \draw  (0,-2.5) ellipse (3.5 and 0.3);


  \draw[red, thick, dashed] (0.5,0) -- (0.5,-5);
  \draw[red, thick, dashed] (0,0) -- (0,-5);


  \fill (0.5,-4) circle (2pt); 
   \node[right] at (0.5,-4) {$x$}; 

    \fill (0,-1) circle (2pt); 
   \node[right] at (0,-1) {$z$}; 

    \fill (2,-2.5) circle (2pt); 
   \node[right] at (2,-2.5) {$y$}; 


   \draw[thick] (0.5,-4) -- (2,-2.5);
    \draw[thick] (0,-1) -- (2,-2.5);

             \end{tikzpicture}

		\label{fig:figure1}
	\end{minipage}
    \caption{The blue cylinder is $\rotom$ and the black revolution is $\D$. The parallel transport is from $x$ to $y$ and then $y$ to $z$ through the black broken light-like geodesic. The dashed lines are world lines $\mu_{a_1}$ and $\mu_{a_2}$ through $x$ and $z$, respectively.}
\end{figure}
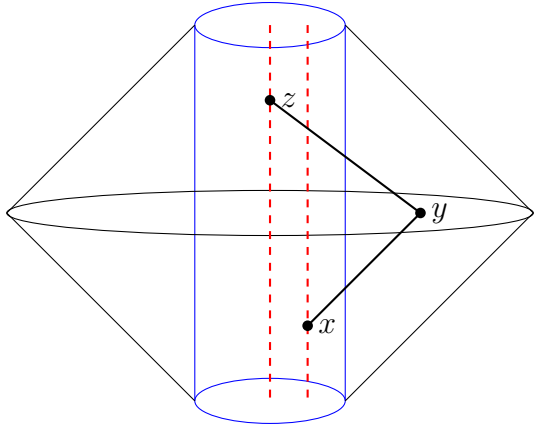

Recall the definition of $\rotom$ given in \ref{eqn: mou}. The set $\rotom$ can be foliated by world lines of the form  $\mu_{a}(s) = (s,a)$, with $a\in B(0,d)$. That is,
$$
  \rotom = \bigcup_{a\in B(0,d)}\mu_{a}((0,T)).
$$
Therefore, for any point $y\in \D$ defined in \eqref{eqn: causal diamond}, there exists $a_1,a_2\in B(0,d)$ such that
\begin{equation}\label{eqn: causal relation}
  y\in J^{+}(\mu_{a_1}(0,T)) \cap J^{-}(\mu_{a_2}(0,T)).
\end{equation}

Define the earliest observation time functions on $\D$, following \cite[Definition 2.1]{KLU}, by
 \begin{align*}
       f_{a}^{+}(y) &= \inf\l( \{s\in (0,T);\ \tau(y,\mu_{a}(s))>0\}\cup\{T\} \r).\\
       f_{a}^{-}(y) &= \sup\l( \{s\in (0,T);\ \tau(\mu_{a}(s),y)>0\}\cup\{0\} \r).
  \end{align*}
Combining \cite[Lemma 2.3]{KLU} and the fact \eqref{eqn: causal relation}, the following result holds.
\begin{lemma}\label{lemma: earliest observation time}
  Let $y\in \D$. Then there exists $a_1,a_2\in B(0,d)$ such that $f_{a_1}^{-}(y),f_{a_2}^{+}(y) \in (0,T)$ and
  $$
     \tau(\mu_{a_1}(f_{a_1}^{-}(y)),y) = 0,   \quad\text{and}\quad    \tau(y,\mu_{a_2}(f_{a_2}^{+}(y))) = 0
  $$
  which implies that there exist past-pointing lightlike geodesic $\gamma_{-}(s)$ and future-pointing lightlike geodesic $\gamma_{+}(s)$ with $\gamma_{\pm}(0) = y$, such that 
  $$
   \mu_{a_1}(f_{a_1}^{-}(y)) \in \gamma_{-}([0,\rho(y,\dot{\gamma}_{-}(0))]),\quad\text{and}\quad
  \mu_{a_2}(f_{a_2}^{+}(y)) \in \gamma_{+}([0,\rho(y,\dot{\gamma}_{+}(0))])
  $$
\end{lemma}

Moreover, define $\E(y)$ as the the set of pairs of past-pointing and future-pointing lightlike vectors at $y$ that are not colinear, that is,
$$
 \mathcal{E}(y) = \{(v,w)\in L^{-}_{y}M\times L^{+}_{y}M: v\neq \lambda w,\text{ for all }\lambda\in\R \}.
$$

For all $(v,w) \in \mathcal{E}(y)$, the geodesic segments $\gamma_{y,v}([0,\rho(y,v)))$ and $\gamma_{y,w}([0,\rho(y,w)))$ intersect only at the point $y$, since there is no closed causal curve on a globally hyperbolic Lorentzian manifold, see \cite{Bernal-Sanchez-2007}. 

\begin{lemma}\label{lemma: observation points}
  For all $y\in \D\b\rotom$, there exist $x,z\in \rotom$ and $(v,w)\in \E(y)$  such that
  $$
    x \in \gamma_{y,v}((0,\rho(y,v)))\quad\text{and}\quad z\in \gamma_{y,w}((0,\rho(y,w))).
  $$
\end{lemma}
\begin{proof}
   By Lemma \ref{lemma: earliest observation time}, there exists $a_1,a_2\in B(0,d)$ such that
   $$
  \tau(\mu_{a_1}(f_{a_1}^{-}(y)),y) = 0\quad\text{and}\quad   \tau(y,\mu_{a_2}(f_{a_2}^{+}(y))) = 0.
   $$
   Set $\tilde{x} := \mu_{a_1}(f_{a_1}^{-}(y))$ and $\tilde{z} := \mu_{a_2}(f_{a_2}^{+}(y))$. Since $0<f_{a_1}^{-}(y), f_{a_1}^{+}(y)<T$, we have $\tilde{x},\tilde{z}\in \rotom$.
   
 By Lemma \ref{lemma: earliest observation time}, there exist lightlike geodesics $\gamma_{y,\tilde{v}}$ and $\gamma_{y,\tilde{w}}$ such that
   $$
    \gamma_{y,\tilde{v}}(\tilde{s}_{1}) = \tilde{x},\quad\gamma_{y,\tilde{w}}(\tilde{s}_2) = \tilde{z},\quad\text{for some}\quad \tilde{s}_1 \in (0,\rho(y,\tilde{v})], \tilde{s}_{2}\in (0,\rho(y,\tilde{w})].
   $$
     Since $\rotom$ is open, there exists $\e >0$ such that
  $$
  \gamma_{y,\tilde{v}}(\tilde{s}_1 -2\e),\gamma_{y,\tilde{w}}(\tilde{s}_2 - 2\e):= z \in\rotom.
  $$

   It may happen that $\gamma_{y,\tilde{v}}$ and $\gamma_{y,\tilde{w}}$ are the same geodesic after reparametrization. In this case, the vectors $\tilde{v}$ and $\tilde{w}$ are linearly dependent.  Choose a sequence $v_{j}\in L^{-}_{y}M$ such that $v_j \to \tilde{v}$ and $v_j\neq \tilde{v}$ for all $j$. Since the null cut function $\rho$ is lower semi-continuous, we have $\rho(y,v_j)\ge \rho(y,\tilde{v})-\e$ for large enough $j$. Then, $\tilde{s}_1-2\e \in (0,\rho(y,v_j)-\e]$, and the points
   $$
   \tilde{x}_{j}:= \gamma_{y,\tilde{v}_j}(\tilde{s}_1-2\e) = \exp_{y}((\tilde{s}_1-2\e)\tilde{v}_j)
   $$ are well-defined, where $\exp_y$ denotes the exponential map on $(M,g)$. Since $\exp_y$ is continuous, we have $\tilde{x}_j\to \tilde{x}$ as $j\to \infty$.

   We now choose $x :=\tilde{x}_j$ for some sufficiently large $j$ such that $x\in \rotom$ and set $v := \tilde{v}_j$. We end the proof by letting $w := \tilde{w}$ and $z$. 
\end{proof}

First, we define the set of \textit{admissible} broken light-like geodesic segments associated to $(v,w)\in\E(y)$ as 
\begin{equation}\label{eqn: admissible broken light-like geodesic segments}
\mathcal{A}(y,v,w) =\left\{
    (x,z) \in \rotom \times \rotom 
    \;\middle|\;
    \begin{aligned}
        &x = \gamma_{y,v}(s'), \ z = \gamma_{y,w}(s'') \text{ for some }\\
        &  s' \in (0,\rho(y,v)), s'' \in (0,\rho(y,w)).
    \end{aligned}
\right\}.
\end{equation}
Then, we define the broken light-ray transform as follows. 
\begin{definition}\label{def: broken light ray transform}
   Suppose $y\in \D$, $(v,w)\in \mathcal{E}(y)$, and $(x,z)\in \mathcal{A}(y,v,w)$. Define $s^{\prime}$ and $s^{\prime\prime}$ by $x = \gamma_{y,v}(s^{\prime})$ and $z = \gamma_{y,w}(s^{\prime\prime})$.
  Let $\xi = -\dot{\gamma}_{y,v}(s^{\prime})\in L_{x}^{+}M$. The broken light-ray transform associated with the connection $A$, along the geodesic segments $\gamma_{x,\xi}([0,s^{\prime}])$ and $\gamma_{y,w}([0,s^{\prime\prime}])$, is defined by
   $$
    S^{A}_{\gamma_{x,\xi}([0,s^{\prime}])\gamma_{y,w}([0,s^{\prime\prime}])} := \mathbf{P}^{A}_{\gamma_{y,w}([0,s^{\prime}])} \circ\mathbf{P}^{A}_{{\gamma}_{x,\xi}([0,s^{\prime\prime}])}.
   $$
\end{definition}

\subsection{Recovery of the connection}
In this section, we prove that the connection can be determined up to a natural gauge transformation given the broken light-ray transform. The argument is inspired by the proofs of \cite[Theorem 5]{CLOP} and \cite[Theorem 4.2]{St-Amant-2024}.

\begin{theorem}\label{thm: inversion of the light-ray transform}
  Let $A$ and $B$ be two Hermitian connections. Suppose   $$
  S^{A}_{\gamma_{x,\xi}([0,s^{\prime}])\gamma_{y,w}([0,s^{\prime\prime}])} = S^{B}_{\gamma_{x,\xi}([0,s^{\prime}])\gamma_{y,w}([0,s^{\prime\prime}])}
  $$ for all $y\in \D$ and for all lightlike geodesic segments $\gamma_{x,\xi}([0,s^{\prime}])$ and $\gamma_{y,w}([0,s^{\prime\prime}])$, and $A=B$ on $\rotom$. Then there exists a gauge transformation $\boldsymbol{\varphi}$ satisfying \eqref{eqn: natural gauge} such that
   $$
  A = B \triangleleft \boldsymbol{\varphi}.
   $$

\end{theorem}

The proof is based on a gauge fixing argument as inspired by \cite[Lemma 3]{CLOP}. Given two Hermitian connections $A$ and $B$, define a $U(n)$ valued function by
$$
  \boldsymbol{\varphi}(y;w,s^{\prime\prime}) = \l(\mathbf{P}^{B}_{\gamma_{y,w}([0,s^{\prime\prime}])} \r)^{-1}\mathbf{P}^{A}_{\gamma_{y,w}([0,s^{\prime\prime}])}, 
$$
where $w\in L^{+}_{y}M$ and $s^{\prime\prime}\in (0,\rho(y,w))$ is a parameter of the lightlike geodesic such that $\gamma_{y,w}(s^{\prime\prime})\in \rotom$.
\begin{lemma}\label{lemma: gauge fixing}
Suppose that for $j=1,2$, we are given vectors $w_j\in L_{y}^{-}M$ and parameters $s^{\prime\prime}_j\in (0,\rho(y,w_j))$ such that $\gamma_{y,w_j}(s^{\prime\prime}_j)\in \rotom$. Under the assumption of Theorem \ref{thm: inversion of the light-ray transform}, we have
$$
   \boldsymbol{\varphi}(y;w_1,s^{\prime\prime}_1) = \boldsymbol{\varphi}(y;w_2,s^{\prime\prime}_2).
$$  
\end{lemma}

\begin{proof}
The assumption that the broken light-ray transforms for $A$ and $B$ coincide implies that
$$
  \mathbf{P}^{A}_{\gamma_{y,w}([0,s^{\prime\prime}])}\mathbf{P}^{A}_{\gamma_{x,\xi}([0,s^{\prime}])} = \mathbf{P}^{B}_{\gamma_{y,w}([0,s^{\prime\prime}])}\mathbf{P}^{B}_{\gamma_{x,\xi}([0,s^{\prime}])},
$$
for all $(v,w)\in \E(y)$ such that $\gamma_{y,v}(s^{\prime}),\gamma_{y,w}(s^{\prime\prime})\in \rotom$, with $s^{\prime}\in (0,\rho(y,v)), s^{\prime\prime}\in (0,\rho(y,w))$, and 
$$
(x,\xi) = (\gamma_{y,v}(s^{\prime}),-\dot{\gamma}_{y,v}(s^{\prime\prime})).
$$
Recall that the inversion of the parallel transport map is obtained by solving the transport equation backwards. Then, the solution of the transport equation forms a group. Rearranging terms, we obtain the identity
\begin{equation}\label{eqn: identity on gauge fixing}
  \l(\mathbf{P}^{B}_{\gamma_{y,w}([0,s^{\prime\prime}])}\r)^{-1} \circ\mathbf{P}^{A}_{\gamma_{y,w}([0,s^{\prime\prime}])} = \mathbf{P}^{B}_{\gamma_{x,\xi}([0,s^{\prime}])}\circ \l(\mathbf{P}^{A}_{\gamma_{x,\xi}([0,s^{\prime}])}\r)^{-1}.
\end{equation}

\begin{claim} There exists $v\in L_{y}^{-}M$ such that $(v,w_j)\in \E(y)$ for $j=1,2$, and the geodesic segment $\gamma_{y,v}(0,\rho(y,v))$ intersects $\rotom$.
\end{claim}

By Lemma \ref{lemma: earliest observation time}, there exists a past-pointing lightlike geodesic $\gamma_{y,\tilde{v}}$ from $y$ to the point $\mu_{a}(f^{-}_{a}(y))$ for some $a\in B(0,d)$, such that
$$
   \mu_a(f_{a}^{-}(y)) \in \gamma([0,\rho(y,\tilde{v})]).
$$
By the lower semi-continuity of the cut-time function $\rho$, for any $\e>0$, there exists a neighborhood $V\subseteq L_{y}^{-}M\b0$ of $\tilde{v}$ such that 
$$
\rho(y,v) > \rho(y,\tilde{v})-\e,\quad\text{for all }v\in V.
$$ 
If $\tilde{v} = \lambda w_j$ for some $j=1,2$ and $\lambda\in \R\b0$, we choose $v\in V$ such that $v$ is not collinear with either $w_1$ or $w_2$. Then $(v,w_j)\in \E(y)$ for $j=1,2$. Moreover, since $\gamma_{y,\tilde{v}}([0,\rho(y,\tilde{v})])$ intersects $\rotom$, the openness of $\rotom$ implies that for sufficiently small $\e>0$, we have 
$$
\gamma_{y,v}(0,\rho(y,v)) \cap \rotom \neq \varnothing
$$ 
This proves the claim.

Then, there exists a parameter $s^{\prime}\in (0,\rho(y,v))$ such that $x := \gamma_{y,v}(s^{\prime})\in \rotom$. Let $\xi := -\dot{\gamma}_{y,v}(s^{\prime})$. By the identity \eqref{eqn: identity on gauge fixing}, we have 
$$
 \l(\mathbf{P}^{B}_{\gamma_{y,w_j}([0,s_{j}^{\prime\prime}])}\r)^{-1}\mathbf{P}^{A}_{\gamma_{y,w_j}([0,s_{j}^{\prime\prime}])} = \mathbf{P}^{B}_{\gamma_{x,\xi}([0,s^{\prime}])} \l(\mathbf{P}^{A}_{\gamma_{x,\xi}([0,s^{\prime}])}\r)^{-1},\quad \text{for }j=1,2.
$$
It follows that 
$$
\boldsymbol{\varphi}(y;w_1,s^{\prime\prime}_1) = \boldsymbol{\varphi}(y;w_2,s^{\prime\prime}_2).
$$
 
\end{proof}

The gauge $\boldsymbol{\varphi}$ defined in Lemma \ref{lemma: gauge fixing} depends only on the point $y$, i.e., it defines a function $\boldsymbol{\varphi}(y)$. Note that the geodesic $\gamma_{y,w}$ depends smoothly on the initial data $(y,w)$. Thus, the transport equations \eqref{eqn: equation of parallel transport} and \eqref{eqn: parallel transport is invertible} depend smoothly on $(y,w)$, which implies that
$$
\boldsymbol{\varphi} \in C^{\infty}(\D;U(n)).
$$
Moreover, it is clear that $\boldsymbol{\varphi}|_{\rotom} = \id$, since $A=B$ on $\rotom$.

\begin{proof}[Proof of Theorem \ref{thm: inversion of the light-ray transform}]
Let $\phi_{t}(y,w)$ denote the geodesic flow starting from $(y,w)$. Then, we have
$$
  \boldsymbol{\varphi}(\gamma_{y,w}(t)) = \l(\mathbf{P}^{B}_{\gamma_{\phi_{t}(y,w)}([0,s^{\prime\prime}])} \r)^{-1}\mathbf{P}^{A}_{\gamma_{\phi_{t}(y,w)}([0,s^{\prime\prime}])}.
$$
For convenience, write 
$$
\gamma(s) := \gamma_{y,w}(s),\quad\text{and}\quad \gamma_{\phi_{t}}(s) := \gamma_{\phi_{t}(y,w)}(s) = \gamma_{y,w}(s+t).
$$ 
We now consider the derivative of $\boldsymbol{\varphi}(\gamma(t))$ with respect to $t$.

By the lower semi-continuity of the null cut function $\rho$, we have that for all $\e>0$, there exists $\delta>0$ such that
\begin{equation}\label{eqn: lower semi-continuity with geodesic flow}
  \rho(\phi_{t}(y,w))\ge \rho(y,w)-\e,\quad \text{for all } t\in(-\delta,\delta).
\end{equation}
We may assume satisfies $\delta<\e$ without loss of generality. Fix such $\e>0$ and corresponding $\delta>0$. Define two parameters
\begin{equation}\label{eqn: domain of two parameters}
  t\in (-\delta/2,\delta/2),\quad s\in (-\delta/2, \rho(y,w)-2\e).
\end{equation}
For all $(t,s)$ in the domain defined by \eqref{eqn: domain of two parameters}, define 
$$
\tilde{U}^{A}(t,s):= \mathbf{P}^{A}_{\gamma_{\phi_t}([0,s])},
$$
which satisfies the transport equation
\begin{equation}
   \begin{aligned}
     \begin{cases}
        \partial_{s}\tilde{U}^{A}(t,s)+\langle A(\gamma(t+s)),\dot{\gamma}(t+s)\rangle \tilde{U}^{A}(t,s) = 0\\
        \tilde{U}^{A}(t,0) = I_n.
    \nonumber
     \end{cases}
    \end{aligned}
\end{equation}

We observe that 
$$
\tilde{U}^{A}(0,s) = \mathbf{P}^{A}_{\gamma ([0,s])}.
$$
Similarly, $\tilde{U}^{B}(t,s)$ is the solution to the corresponding transport equation for the connection $B$. By Lemma \ref{lemma: gauge fixing} and the definition of $\tilde{U}^{A}$ and $\tilde{U}^B$, we have
\begin{equation}\label{eqn: gauge under the geodesic flow}
  \boldsymbol{\varphi}(\gamma(t)) = \tilde{U}^{B}(t,s^{\prime\prime})^{-1}\tilde{U}^{A}(t,s^{\prime\prime}) = \tilde{U}^{B}(t,s^{\prime\prime}-t)\tilde{U}^{A}(t,s^{\prime\prime}-t),
\end{equation}
where the second equality follows Lemma \ref{lemma: gauge fixing}, since the endpoints of the geodesic segments lie in $\rotom$ and remain strictly before the first point for $t\in (-\delta/2,\delta/2)$.

We perform a change of variables by defining
$$
U^{A}(t,s) := \tilde{U}^{A}(t,s-t),
$$
which satisfies the transport equation
\begin{equation}\label{eqn: fundamental solution of the parallel transport}
  \begin{aligned}
    \begin{cases}
     \partial_{s}U^{A}(t,s) + \langle A(\gamma(s)),\dot{\gamma}(s)\rangle U^{A}(t,s) = 0,\\
     U^{A}(t,t)=I_n,
    \end{cases}
  \end{aligned}
\end{equation}
where, for fixed $t\in (-\delta/2,\delta/2)$, the parameter $s$ lies in the interval $(-\delta/2+t,\rho(y,w)-2\e +t)$. In fact, the domian of the equation \eqref{eqn: fundamental solution of the parallel transport} can be extended to
\begin{equation}\label{eqn: extended domain of two parameters}
  t\in (-\delta/2,\delta/2),\quad s\in (-\delta,\rho(y,w)-\e).
\end{equation}
Thus, the map $U^{A}(t,s)$ is a smooth two-parameter map in the Lie group $U(n)$, with $(t,s)$ in the domain given by \eqref{eqn: extended domain of two parameters}. Therefore, the order of partial derivatives can be interchanged.
\begin{equation}\label{eqn: change of order of the derivatives of two-parameter map}
  \partial_{t}\partial_{s}U^{A}(t,s) = \partial_{s}\partial_{t}U^{A}(t,s).
\end{equation}
Moreover, by \eqref{eqn: gauge under the geodesic flow}, we have
\begin{equation}\label{eqn: gauge under the geodesic flow with end point invariant}
 \boldsymbol{\varphi}(\gamma(t)) = U^{B}(t,s^{\prime\prime})^{-1}U^{A}(t,s^{\prime\prime})
\end{equation}

Letting $s=t$ in equation \eqref{eqn: fundamental solution of the parallel transport}, we obtain 
\begin{equation}\label{eqn: the derivative of U(t,t) on s parameter}
   \partial_{s}U^{A}(t,t) = -\langle A(\gamma(t)),\dot{\gamma}(t))\rangle.
\end{equation}
Next, we differentiate the equation \eqref{eqn: fundamental solution of the parallel transport} with respect to $t$. Using the commutation of partial derivatives of the two-parameter map established in  \eqref{eqn: change of order of the derivatives of two-parameter map}, we find that
\begin{equation}\label{eqn: transport equation of the first order deriavtive}
  \partial_{s}\partial_{t} U^{A}(t,s) + \langle A(\gamma(s)),\dot{\gamma}(s)\rangle \partial_{t}U^{A}(t,s) = 0.
\end{equation}
The initial condition $U^{A}(t,t)= I_n$, together with equation \eqref{eqn: the derivative of U(t,t) on s parameter}, implies that
\begin{equation}\label{eqn: initial value of the first order derivative}
  \partial_{t}U^{A}(t,t) = -\partial_{s}U^{A}(t,t) = \langle A(\gamma(t)),\dot{\gamma}(t)\rangle.
\end{equation}

Then, $\partial_{t}U^{A}(t,s)$ satisfies the transport equation \eqref{eqn: transport equation of the first order deriavtive} with initial condition given by \eqref{eqn: initial value of the first order derivative}. Recall that $U^{A}(t,s)$, as the solution of \eqref{eqn: fundamental solution of the parallel transport} with initial data $I_n$, is the fundamental solution to the transport equation. Thus, by the uniqueness of linear ordinary differential equations, we obtain
\begin{equation}\label{eqn: solution of the first order derivative}
 \partial_{t}U^{A}(t,s) = U^{A}(t,s)\langle A(\gamma(t)),\dot{\gamma}(t)\rangle.
\end{equation}

Moreover, we compute the derivative of the inverse matrix $\partial_{t}\l( U^{A}(t,s)^{-1}\r)$. Differentiating both sides of the identity
$$
  U^{A}(t,s)^{-1} U^{A}(t,s) = I_n
$$
with respect to $t$, we obtain
$$
  \partial_{t} \l(U^{A}(t,s)^{-1} \r) U^{A}(t,s) + U^{A}(t,s)^{-1}\partial_{t}U^{A}(t,s) = 0.
$$
By the expression of the partial derivative $\partial_{t}U^{A}(t,s)$ given in \eqref{eqn: solution of the first order derivative}, we have
\begin{align}
 \notag  \partial_{t}\l(U^{A}(t,s)^{-1} \r) &= -U^{A}(t,s)^{-1}\partial_{t}U^{A}(t,s)U^{A}(t,s)^{-1}\\
  \notag &=-U^{A}(t,s)^{-1} U^{A}(t,s)\langle A(\gamma(t)),\dot{\gamma}(t)\rangle U^{A}(t,s)^{-1}\\
\label{eqn: the first order derivative of the inverse}   &= -\langle A(\gamma(t)),\dot{\gamma}(t)\rangle U^{A}(t,s)^{-1}.
\end{align}

We observe that both \eqref{eqn: solution of the first order derivative} and \eqref{eqn: the first order derivative of the inverse} hold similarly for $U^{B}(t,s)$. Combining these with equation \eqref{eqn: gauge under the geodesic flow with end point invariant}, we compute the derivative of the gauge
\begin{align*}
\partial_{t}\boldsymbol{\varphi}(\gamma(t)) &= \partial_{t}\l(U^{B}(t,s)^{-1} \r) U^{A}(t,s) + U^{B}(t,s)^{-1}\partial_{t}U^{A}(t,s)\\
&= - \langle B(\gamma(t)),\dot{\gamma}(t)\rangle U^{B}(t,s)^{-1}U^{A}(t,s) + U^{B}(t,s)^{-1}U^{A}(t,s)\langle A(\gamma(t)),\dot{\gamma}(t)\rangle\\
&= - \langle B(\gamma(t)),\dot{\gamma}(t)\rangle \boldsymbol{\varphi}(\gamma(t)) + \boldsymbol{\varphi}(\gamma(t))\langle A(\gamma(t)),\dot{\gamma}(t)\rangle\\
&= \langle(\boldsymbol{\varphi}A- B\boldsymbol{\varphi})(\gamma(t)),\dot{\gamma}(t)\rangle.
\end{align*}
Evaluating at $t =0$, we obtain
$$
 \langle d\boldsymbol{\varphi}(y), w\rangle = \langle \boldsymbol{\varphi}(y)A(y)- B(y)\boldsymbol{\varphi}(y),w\rangle.
$$
The identity also holds for all $\tilde{w}$ in a small neighborhood $W\subseteq L_{y}^{+}M$ of $w$. Since the linear span of vectors in $W$ equals the tangent space $T_yM$, it follows that 
$$
  d\boldsymbol{\varphi} = \boldsymbol{\varphi}A-B\boldsymbol{\varphi} \quad \text{on } \D.
$$
Left-multiplying both sides by $\boldsymbol{\varphi}^{-1}$, we conclude that 
$$
 A = B\triangleleft\boldsymbol{\varphi}.
$$
This completes the proof.
\end{proof}

\section{Microlocal analysis of linear waves }\label{sec: linear waves}
\subsection{Connection wave operator}
In this section, we derive the expression of the connection wave operator $\Box_{g,A}$ in local coordinates $(x^{0},x^{1},x^{2},x^{3})$ on the Lorentzian manifold $(M,g)$. The metric tensor is given by $g = g_{ij}dx^i dx^j$, with inverse $g^{ij}$. Let $|g|$ denote the absolute value of the determinant of the matrix $(g_{ij})$.

The $L^{2}-$ inner product on $L^2(M;E)$ is given by 
$$
  \langle u,v\rangle_{L^{2}(M;E)} = \int_{M}\langle u,v\rangle_{E} |g|^{1/2}dx.
$$
Note that the covariant derivative $d+A$ on the vector bundle $E$ maps $C^{\infty}(M;E)$ to $C^{\infty}(M;T^{\ast}M\otimes E)$, which is the space of $1-$forms taking values on $E$. Given compactly supported smooth $1-$forms $\alpha,\beta \in C_{c}^{\infty}(M;E\otimes T^{\ast}M)$, with local expressions $\alpha = \alpha_{i}dx^{i}$ and $\beta = \beta_{j}dx^j$, the $L^{2}-$inner product in $T^{\ast}M\otimes E$ is 
$$
  \langle \alpha,\beta\rangle_{L^{2}(M;T^{\ast}M\otimes E)} = \int_{M} g^{ij}\langle\alpha_{i},\beta_{j}\rangle_{E}|g|^{1/2} dx
$$

Recall that $A = A_{i}dx^{i}$ with each $A_{i} \in C^{\infty}(M;\mathfrak{u}(n))$, where $\mathfrak{u}(n)$ denotes the Lie algebra of skew-Hermitian $n\times n$ matrices. Given $u \in C_{c}^{\infty}(M;E)$, the covariant derivative has the local expression
$$
(d+A) u = (\partial_{x^i}u + A_{i}u)dx^{i}.
$$
Let $\beta = \beta_i dx^{i}\in C_{c}^{\infty}(M;T^{\ast}M\otimes E)$. Then the $L^2-$inner product of $(d+A)u$ and $\beta$ in $L^{2}(M;T^{\ast}M\otimes E)$ is given by
$$
  \langle (d+A)u,\beta\rangle_{L^{2}(M;T^{\ast}M\otimes E)} = \int_{M} g^{ij}\langle \partial_{x^i}u+ A_{i}u, \beta_{i}\rangle_{E}|g|^{1/2}dx.
$$
Using integration by parts and the fact that $u$ and $\beta$ have compact supports, we obtain
\begin{align*}
\langle u, (d+A)^{\ast}\beta\rangle_{L^{2}(M;E)} &= \langle (d+A)u,\beta\rangle_{L^{2}(M;T^{\ast}M\otimes E)}  \\
&= -\int_{M} |g|^{1/2} \langle u, |g|^{-1/2}\partial_{x^i}( |g|^{1/2}g^{ij}\beta_{i})+g^{ij}A_{i}\beta_{i}\rangle_{E}dx.
\end{align*}
Thus, for any $\beta = \beta_{i}dx^{i}\in C^{\infty}(M;T^{\ast}M\otimes E)$, the adjoint of the covariant derivative $d+A$ is given in local coordinates by
$$
 (d+A)^{\ast}\beta = -|g|^{-1/2}\partial_{x^{i}}(|g|^{1/2}g^{ij}\beta_i) - g^{ij}A_i\beta_i.
$$
Consequently, the connection wave operator $\Box_{g,A} = (d+A)^{\ast}(d+A)$ takes the form, see \cite[Section 2.1]{CLOP},
\begin{multline*}
  \Box_{g,A} u  =  (d+A)^{\ast}(d+A)u\\
  = -|g|^{-1/2}\partial_{x^i}(|g|^{1/2}g^{ij}\partial_{x^j}u) -2g^{ij}A_i \partial_{x^j}u - |g|^{-1/2}\partial_{x^i}(|g|^{1/2}g^{ij}A_j)u-g^{ij}A_{i}A_{j}u.
\end{multline*}

To apply the framework of microlocal analysis developed in \cite{DH,H}, it is necessary to consider the operator $\Box_{g,A}$ acting on half-density valued functions. The half-density bundle $\Omega^{1/2}$ over $M$ is a one-dimensional complex vector bundle locally trivialized by $|g|^{1/4}$. As a result, a section  $\tilde{u}\in H^{s}(M;E)$ corresponds to a half-density valued function
\begin{equation}\label{eqn: correspond to half-density valued function}
   u :=\tilde{u}|g|^{1/4} \in H^{s}(M;E\otimes\Omega^{1/2}).
\end{equation}
and vice versa.

The conjugated wave operator $P(\cdot):= \frac{1}{2} |g|^{1/4}\Box_{g,A}(|g|^{-1/4} \cdot)$, acting on half-density valued distributions, takes the following form in local coordinates.
\begin{equation}\label{eqn: conjugated wave operator}
\begin{aligned}
  P =& -\frac{1}{2}g^{ij}\partial_{x^i}\partial_{x^j} -\frac{1}{2}\partial_{x^i}g^{ij}\partial_{x^j}-g^{ij}A_{i}\partial_{x_j}\\
  &-\frac{1}{2}|g|^{-1/4}\partial_{x^i}(g^{ij}|g|^{1/2}\partial_{x^j}(|g|^{-1/4}))-g^{ij}A_i\partial_{x^j}(\log|g|^{-1/4})\\
  &-\frac{1}{2}|g|^{-1/2}\partial_{x^i}(|g|^{1/2}g^{ij}A_j) - \frac{1}{2}g^{ij}A_iA_j.
\end{aligned}
\end{equation}

Accordingly, the semi-linear wave equation \eqref{eqn: connection wave equation with a source} can be rewritten for $u\in H^{k+1}(M;E\otimes \Omega^{1/2})$ with $k\ge 4$ as follows.
\begin{equation}\label{eqn: connection wave equation of half-density valued}
 \begin{aligned}
  \begin{cases}
Pu + \frac{1}{2}\l||g|^{-1/4}u\r|^{2} u = \frac{1}{2}|g|^{1/4} f,\quad \text{on}&(0,T)\times N,\\
   u(0,x^{\prime})=0, \partial_{t}(0,x^{\prime}) = 0,\quad &x^{\prime}\in N.
   \end{cases}
   \end{aligned}
\end{equation}
where $\frac{1}{2}|g|^{1/4}f$ is a half-density valued source in $H^{k}(M;E\otimes \Omega^{1/2})$.

Let $(x,\xi)\in T^{\ast}M$. The principal symbol of $P$, in the sense of \cite[Definition 3.8]{Grigis-Sjostrand}, is
$$
 \sigma[P](x,\xi) =\frac{1}{2}g^{ij}\xi_i\xi_j.
$$

The subprincipal symbol of $P$, defined in a coordinate invariant manner as in \cite[Theorem 18.1.33]{H3}, is given by
$$
 \sigma_{\sub}[P](x,\xi) = \imath^{-1} g^{ij}A_{i}\xi_j.
$$
Here $\imath = \sqrt{-1}$.

\subsection{Lagrangian distributions}
To study the microlocal structure of the solution to  the linear wave equation with a connection,
\begin{equation}\label{eqn: linear wave equation with a connection on half-density}
 \begin{aligned}
   \begin{cases}
 Pu =f,\quad&\text{on } (0,T)\times N,\\
  u(0,x^{\prime}) = 0, \partial_{t}u(0,x^{\prime} )=0\quad &x^{\prime}\in N
  \end{cases}
  \end{aligned}
\end{equation}
where $f$ and $u$ are sections on the vector bundle $E\otimes \Omega^{1/2}$, we begin by briely reviewing the the theory of Lagrangian distributions in the sense of \cite{H,H4}.

Let $X$ be a smooth manifold of dimension $n$. The local coordinates on the cotangent bundle $T^{\ast}X$ are denoted by $(x^1,\dots,x^n;\xi_1,\dots,\xi_n)$. Let $\Lambda\subseteq T^{\ast}X\b 0$ be a smooth conic Lagrangian submanifold, as defined in \cite[Definition 21.1.8, 21.2.5]{H3}.

Suppose that $\phi(x,\theta)$ is a nondegenerate phase function in the sense of \cite[Definition 21.2.15]{H3}, which parametrizes $\Lambda$ in a conic neighborhood $\Gamma\subseteq X\times\R^{N}$. Then,
$$
 \{(x,d_x\phi(x,\theta)): d_{\theta}\phi(x,\theta)=0, (x,\theta)\in\Gamma\} \subseteq \Lambda.
$$
Recall that $E$ is a Hermitian vector bundle and $\Omega^{1/2}$ denotes the half-density bundle over $X$. We define $E\otimes\Omega^{1/2}$-valued Lagrangian distributions as follows. 
\begin{definition}\label{def: Lagrangian distributions}
Let $\Lambda\subseteq T^{\ast}X\b 0$ be a conic Lagrangian submanifold. The space $I^{m}(X;\Lambda;E\otimes \Omega^{1/2})$ consists of $E\otimes \Omega^{1/2}$-valued distributions $u\in \mathcal{D}^{\prime}(X)$ associated with $\Lambda$ satisfying $\WF(u)\subseteq \Lambda$. In local coordinates $x\in U\subseteq X$, the distibution $u$ is given by an oscillatory integral of
  $$
     u(x) = (2\pi)^{\frac{n-2N}{4}}\int_{\R^N} e^{\imath \phi(x,\theta)}a(x,\theta)d\theta,\quad x\in U,
  $$
  where 
  \begin{itemize}
      \item $\phi(x,\theta)$ is a nondegenerate phase function paramerizing $\Lambda$.
      \item  $a(x,\theta)\in S^{m+\frac{n}{4}-\frac{N}{2}}(U\times \R^{N}\b 0;E\otimes\Omega^{1/2})$ is a classical symbol in the sense of \cite[Definition 18.1.1]{H3}.
  \end{itemize}
\end{definition}
The Hessian of the phase function $\phi(x,\theta)$ is the nondegenerate symmetric matrix
$$
 \Phi = \begin{pmatrix}  
 \phi_{xx} & \phi_{x\theta} \\ 
 \phi_{\theta x} & \phi_{\theta\theta} 
 \end{pmatrix}.
$$
Denote by $\sgn\Phi$ the signature of $\Phi$, and define the associated density 
\begin{equation}\label{eqn: density}
d_C = |d\xi| \ |\det\Phi|^{-1},
\end{equation}
where $\xi$ denotes the fiber coordinate in $T^{\ast}X$.

The principal symbol of the Lagrangian distribution $u$ is define by
\begin{multline}\label{eqn: principal symbol}
 \sigma[u](x,d_x\phi) = e^{\frac{\imath\pi }{4}\sgn\Phi}a(x,d_x\phi)\sqrt{d_C} \\
 \in S^{m+\frac{n}{4}}(\Lambda;E\otimes\Omega_{1/2}\otimes L) /S^{m+\frac{n}{4}-1}(\Lambda;E\otimes\Omega_{1/2}\otimes L), 
\end{multline}
where $\Omega_{1/2}$ is the half-density bundle over $\Lambda$, and $L$ denotes the Keller-Maslov line bundle, see \cite[pp. 13-16]{H4}. For simplicity, we write $S^{m+\frac{n}{4}}(\Lambda)$ to denote the symbol space $S^{m+\frac{n}{4}}(\Lambda;E\otimes\Omega_{1/2}\otimes L)$.

 Kurylev-Lassas-Uhlmann \cite{KLU} developed a framework for using conormal waves to detect singularities generated by nonlinear wave interactions. Conormal waves are locally modeled by conormal distributions, which form a special subclass of Lagrangian distributions. We briefly recall the theory of conormal distributions in the sense of \cite[Section 18.2]{H3}.

Let $S\subseteq X$ be an $r$-dimensional smooth submanifold. The conormal bundle of $S$ is defined by
$$
  N^{\ast}S = \{(x,\xi)\in T^{\ast}X\b0: x\in S,\xi|_{T_{x}S} = 0\}.
$$
In fact, $N^{\ast}S$ is a Langrangian submanifold of $T^{\ast}X\b 0$.

In local slice coordinates, the submanifold $S$ and its conormal bundle $N^{\ast}S$ are respectively written as 
$$
S =\{(x^{1},\dots,x^{r})\},\quad N^{\ast}S = \{ (x^1,\dots,x^r;\xi_{r+1},\dots,\xi_n)\}.
$$
Let us write the local coordinates on $X$ as $x = (x^{\prime},x^{\prime\prime})$, where $x^{\prime} = (x^1,\dots,x^r)$ and $x^{\prime\prime} = (x^{r+1},\dots,x^{n})$, and similarly write the fiber variables $\xi =(\xi^{\prime},\xi^{\prime\prime})$, with $\xi^{\prime} = (\xi_1,\dots,\xi_r)$ and $\xi^{\prime\prime} = (\xi^{r+1},\dots,\xi^{n})$. 

\begin{definition}\label{def: conormal distributions}
The space $I^{m}(N^{\ast}S;E\otimes \Omega^{1/2})$ consists of $E\otimes \Omega^{1/2}$-valued distributions $u\in\mathcal{D}^{\prime}(X)$ associated to the conormal bundle $N^{\ast}S$, with wavefront set $\WF(u) \subseteq N^{\ast}S$, and locally of the form
  $$
    u(x) = (2\pi)^{\frac{n-2r}{4}} \int_{\R^{n-r}} e^{\imath x^{\prime\prime}\xi^{\prime\prime}}a(x^{\prime},\xi^{\prime\prime})d\xi^{\prime\prime}, 
  $$
  where $a\in S^{m+\frac{n}{4}-\frac{n-r}{2}}(\R^{r}\times(\R^{n -r}\b 0); E\otimes \Omega^{1/2})$ is a classical symbol.
\end{definition}
In fact, the conormal distribution space $I^{m}(N^{\ast}S;E\otimes\Omega^{1/2})$ coincides with the Lagrangian distribution space $I^{m}(X;N^{\ast}S;E\otimes \Omega^{1/2})$. The principal symbol of a distribution $u$ in this space is given by 
$$
 \sigma[u](x^{\prime},\xi^{\prime\prime}) = a(x^{\prime},\xi^{\prime\prime})|dx^{\prime}|^{1/2}|d\xi^{\prime\prime}|^{1/2}\in S^{m+\frac{n}{4}}(N^{\ast}S;\Omega_{1/2}) /S^{m+\frac{n}{4}}(N^{\ast}S;\Omega_{1/2}),
$$
where $\Omega_{1/2}$ denotes the half-density bundle over $N^{\ast}Y$. Similarly, we will write $S^{m+\frac{n}{4}}(N^{\ast}S)$ to denote $S^{m+\frac{n}{4}}(N^{\ast}S;\Omega_{1/2})$. We remark that the Keller-Maslov line bundle over a conormal bundle $N^{\ast}S$ is trivial by \cite[Theorem 3.3.4]{H}.

\subsection{Paired Lagrangian distributions}
To analyze the microlocal structure of wave equations with source terms, such as \eqref{eqn: linear wave equation with a connection on half-density}, Melrose-Uhlmann \cite{MU} and Guillemin-Uhlmann \cite{GU81} developed the theory of Intersecting Paired Lagrangian distributions.

Let $\Lambda_0\subseteq T^{\ast}X\b0$ be a conic Lagrangian submanifold, and let $\Lambda_1\subseteq T^{\ast}X\b 0$ be a conic Lagrangian submanifold with boundary. We say that $(\Lambda_0,\Lambda_1)$ is an \textit{intersecting pair of Lagrangian submanifolds}, in the sense of \cite[Definition 1.1]{MU}, if $\Lambda_0$ and $\Lambda_1$ intersect cleanly, that is,
$$
  \Lambda_0 \cap \Lambda_1 =\partial\Lambda_1,\quad T_{\lambda}(\Lambda_0)\cap T_{\lambda}(\Lambda_1) = T_{\lambda}(\partial\Lambda_1),\quad \lambda\in \partial\Lambda_1
$$

Following \cite[Definition 3.1]{MU}, we define the space of intersecting paired Lagrangian distributions.
\begin{definition}
Let $(\Lambda_0,\Lambda_1)$ be an intersecting pair of Lagrangian submanifolds. The space 
$$
I^{m}(X;\Lambda_0,\Lambda_1;E\otimes \Omega_{1/2})
$$
consists of $E\otimes \Omega^{1/2}$-valued distributions $u\in \mathcal{D}^{\prime}(X)$, with wavefront sets $\WF(u)\subseteq \Lambda_0\cup\Lambda_1$, of the form
   $$
       u = u_0+u_1+\sum_{j}F_j v_j,
   $$
   where 
   \begin{itemize} 
     \item $u_0 \in I^{m-1/2}(X;\Lambda_0;E\otimes\Omega^{1/2})$ and $u_1\in I^{m}(X;\Lambda_1\b\partial\Lambda_1;E\otimes \Omega^{1/2})$.

     \item Each $F_j$ is a zeroth order elliptic Fourier integral operator associated with a homogeneous canonical transformation $\chi^{-1}_j$, where $\partial\Lambda_1 \subseteq \bigcup_{j}V_j$ and $\chi_j: V_j \to T^{\ast}\R^{n}$.

     \item The sum $\sum_{j}F_{j}v_j$ is locally finite.

     \item Each $v_j$ is a distributional half-density locally of the form
     $$
    v_j(x) = \int_{0}^{+\infty}\int_{\R^{n}}e^{\imath((x^1-s)\xi_1+x^{\prime}\xi^{\prime})}a_{j}(s,x,\xi)d\xi ds,
     $$
     where $ a_j\in S^{m+\frac{1}{2}-\frac{n}{4}}(\R^{n+1}\times (\R^{n}\b 0); E\otimes\Omega^{1/2})
     $ is a classical symbol.
   \end{itemize}
\end{definition}

Given $u\in I^{m}(X;\Lambda_0,\Lambda_1;E\otimes\Omega^{1/2})$, it follows from \cite[Proposition 4.1]{MU} that for any zeroth  order pseudodifferential operators $B_0,B_1$ satisfying 
$$
\WF(B_0) \cap \Lambda_1 = \varnothing, \quad\text{and}\quad \WF(B_1) \cap \Lambda_0 = \varnothing,
$$
we have the microlocal regularity properties
$$
  B_0u \in I^{m-1/2}(X;\Lambda_0\b\partial\Lambda_1;E\otimes\Omega^{1/2}),\quad B_1 u \in I^{m}(X;\Lambda_1\b\partial\Lambda_0;E\otimes\Omega^{1/2}).
$$
Based on this, \cite[(4.4),(4.5)]{MU} defines the principal symbol of $u$ away from the intersection $\partial\Lambda_1$ as
\begin{align*}
   \sigma[u](\lambda_0) &= \sigma[B_0 u](\lambda_0)/\sigma[B_0](\lambda_0), \quad \lambda_0\in\Lambda_0\b\partial\Lambda_1,\\
    \sigma[u](\lambda_1) &= \sigma[B_1 u](\lambda_1)/\sigma[B_1](\lambda_1), \quad \lambda_1\in\Lambda_1\b\partial\Lambda_1,
\end{align*}
provided that each $B_j$ is elliptic near $\lambda_j$ for $j=0,1$. 
Moreover, the principal symbol $\sigma[u]_{\Lambda_1\b\partial\Lambda_1}$ extends smoothly up to $\partial\Lambda_1$.

On the intersection $\partial\Lambda_1$, there exists a symbol transition map $\mathscr{R}$ defined in \cite[(4.7),(4.12)]{MU}, such that
\begin{align*}
  \mathscr{R}: S^{m-\frac{1}{2}+\frac{n}{4}}(\Lambda_0\b\partial\Lambda_1) &\longrightarrow S^{m+\frac{n}{4}}(\Lambda_1\b\partial\Lambda_1)\\
  \sigma[u]|_{\Lambda_0\b\partial\Lambda_1} &\longmapsto  \sigma[u]|_{\partial\Lambda_1}.
\end{align*}
More precisely, for all $\lambda \in \partial\Lambda_1$, there exists a conic neighborhood $V$ in $T^{\ast}X \b 0$ and a homogeneous canonical transformation $\chi$ such that
\begin{itemize}
  \item $\chi(V\cap \Lambda_0) \subseteq \tilde{\Lambda}_0:=\{(0,\tilde{\xi} );\tilde{\xi}\in\R^{n}\}$.
  \item $\chi(V\cap \Lambda_1)\subseteq \tilde{\Lambda}_1=\{(\tilde{x}^{1},0;0,\tilde{\xi}^{\prime});\tilde{x}^{1}\ge 0,\tilde{\xi}^{\prime}\in \R^{n-1}\}$.
  \item $\chi(\lambda) \in \tilde{\Lambda}_0 \cap \tilde{\Lambda}_1 = \partial\tilde{\Lambda}_1$.
\end{itemize}
In this coordinate, the principal symbol of $u$ on $\tilde{\Lambda}_0$ is written as
$$
   \sigma[u](0,\tilde{\xi}) = a(0,\tilde{\xi})|d\tilde{\xi}|^{1/2},
$$
where $a\in S^{m-1/2-n/4}(\tilde{\Lambda}_0\b\partial\tilde{\Lambda}_1;E)$ and $|d\tilde{\xi}|^{1/2}$ is a trivilization of the half-density bundle on $\tilde{\Lambda}_0$. Generally, $a$ is singular on $\partial\tilde{\Lambda}_1$. Thus,
\begin{equation}\label{eqn: mapping property of R}
 \mathscr{R}(\sigma[u])(0;0,\tilde{\xi}^{\prime}) = \imath\l.(\tilde{\xi}_1 a(0,\tilde{\xi}^{\prime}))\r|_{\tilde{\xi}_1 = 0} |d\tilde{x}^{1}|^{1/2}|d\tilde\xi^{\prime}|^{1/2}.
\end{equation}
By \cite[Proposition 3.2]{MU}, the definition of $\mathscr{R}$ does not depend on the choice of the homogeneous canonical transformation $\chi$.

\subsection{Distorted plane waves}
In this section, we construct distorted plane waves following \cite[Section 3.2.3]{KLU}, using the theory of intersecting paired Lagrangian distributions.

Recall the conjugated wave operator $P$ defined in \eqref{eqn: conjugated wave operator}. The characteristic set $\Char(P)$ coincides with the null cotangent bundle $L^{\ast}M$.

Let $\Lambda_0\subseteq T^{\ast}M\b 0$ be a conic Lagrangian submanifold such that the Hamiltonian vector field $H_P$ is nowhere tangent to $\Lambda_0$.  Denote by $\Lambda_1$ the future flowout of $\Lambda_0 \cap \Char(P)$ under the Hamiltonian vector field $H_P$. Then $\Lambda_1\subseteq T^{\ast}M\b 0$ is a conic Lagrangian submanifold with boundary 
$$
\partial\Lambda_1 = \Lambda_0 \cap \Char(P).
$$
In this setting, the pair $(\Lambda_0,\Lambda_1)$ forms an intersecting pair of Lagrangian submanifolds.

For such a pair $(\Lambda_0,\Lambda_1)$, the following result from \cite[Proposition 6.6]{MU} describes the microlocal structure of the solutions to the equation \eqref{eqn: linear wave equation with a connection on half-density}.
\begin{theorem}\label{thm: paramatrix of the wave equations}
   Let $f\in I^{\mu+3/2}(M;\Lambda_0;E\otimes\Omega^{1/2})$ be the source term in the wave equation \eqref{eqn: linear wave equation with a connection on half-density}. Then, the solution $u$ belongs to the space
   $$
  u\in I^{\mu}(M;\Lambda_0,\Lambda_1; E\otimes\Omega^{1/2}),
   $$
   and its principal symbol $\sigma[u]$ satisfies the transport equation
   \begin{equation}\label{eqn: transport equation for principal symbol}
     \begin{aligned}
       \begin{cases}
        (\mathscr{L}_{H_P}+\imath \sigma_{\sub}[P])\sigma[u] = 0\quad&\text{on }\Lambda_1\b \Lambda_0,\\
        \sigma[u] = \mathscr{R}(\sigma[P]^{-1}\sigma[f]) \quad&\text{on }\partial\Lambda_1,
        \end{cases}
     \end{aligned}
     \end{equation}
     where $\mathscr{L}_{H_P}$ denotes the Lie derivative along the Hamilton vector field $H_P$.
\end{theorem}

Distorted plane waves emanate from a point source on $x_0$ and propagate along the light cone. Fix an auxiliary Riemannian metric $g^{+}$ on $M$ and write $\|\cdot\|_{g^{+}}$ for the corresponding norm. Let $(x_0,\zeta_0)\in LM$ and let $s_0>0$ be sufficiently small. Define a neighborhood of $\zeta_0$ on the sphere of radius $\|\zeta_0\|_{g^{+}}$ as
$$
  \mathcal{V}_{x_0,\zeta_0,s_0} = \{\eta\in T_{x_0}M: \| \eta\|_{g^+} = \|\zeta_0 \|_{g^{+}} , \|\eta-\zeta_0\|_{g^+}< s_0\}.
$$

For a tangent vector $\eta\in T_{x_0}M$, let $\eta^{\flat}\in T^{\ast}_{x_0}M$ denote the corresponding covector given by $\eta^{\flat}_i = g_{ij}\eta^{j}$. Conversely, for $\xi\in T^{\ast}_{x_0}M$, the associated vector $\xi^{\sharp}\in T_{x_0}M$ is given by $(\xi^{\sharp})^{j}=g^{ij}\xi_{i}$.

Consider a source of the linear wave equation \eqref{eqn: connection wave equation of half-density valued} with wavefront set
$$
 \Sigma(x_0,\zeta_0,s_0) = \{(x_0,r\eta^{\flat})\in T^{\ast}M: \eta\in \mathcal{V}_{x_0,\zeta_0,s_0}, r\in \R\b 0\}.
$$
Moreover, define
$$
\mathcal{W}_{x_0,\zeta_0,s_0} := \mathcal{V}_{x_0,\zeta_0,s_0}\cap LM.
$$ 
 Let $\Lambda(x_0,\zeta_0,s_0)$ denote the flowout of $\Sigma \cap \Char(P)$, that is, 
$$
 \Lambda(x_0,\zeta_0,s_0)=\{(\gamma_{x_0,\eta}(t),r\dot{\gamma}_{x_0,\eta}(t)^{\flat})\in T^{\ast}M; \eta \in \W_{x_0,\zeta_0,s_0}, t\in [0,+\infty), r\in \R\b 0\}.
$$
For convenience, we write
$$
 \Sigma:= \Sigma(x_0,\zeta_0,s_0),\quad\text{and}\quad \Lambda:=\Lambda(x_0,\zeta_0,s_0).
$$

Let the source term of the linear wave equation \eqref{eqn: connection wave equation of half-density valued} be $f\in I^{\mu+3/2}(M;\Sigma;E\otimes\Omega^{1/2})$. Then the solution $u$ lies in the space 
$$
I^{\mu}(M;\Sigma,\Lambda;E\otimes\Omega^{1/2}).
$$


Let the flowout from $x_0$ be 
\begin{equation}\label{eqn: flowout submanifold}
  K(x_0,\zeta_0,s_0) = \{\gamma_{x_0,\eta}(s)\in M: \eta \in \mathcal{W}_{x_0,\zeta_0,s_0},  s\in (0,\mathcal{T}(x_0,\eta))\}.
\end{equation}
For convenience, we denote $K:= K(x_0,\zeta_0,s_0)$. For any $y\in K$, there exist $\eta\in \mathcal{W}_{x_0,\zeta_0,s_0}$ and $r>0$ such that $y=\gamma_{x_0,\eta}(r)$.

The following result from \cite[Lemma 3.1]{KLU} shows that $u$ is locally a conormal distribution near each such point before the first cut point.

\begin{lemma}\label{lemma: locally conormal distribution}
Let $y = \gamma_{x_0,\eta}(r) \in K$. If $0<r< \rho(x_0,\eta)$, then there exists a neighborhood $V_0 \subseteq M$ of $y$ such that $K\cap V_0$ is a smooth submanifold of codimension one, and $N^{\ast}(K\cap V_0) \subseteq \Lambda$. Then, the solution $u$ of \eqref{eqn: linear wave equation with a connection on half-density} with $f\in I^{\mu+3/2}(M;\Sigma;E\otimes \Omega^{1/2})$ is a conormal distribution restricted on $V_0$, that is, 
$$
u|_{V_0}\in I^{\mu}(N^{\ast}K;E\otimes\Omega^{1/2}).
$$ 
\end{lemma}
Away from the source point $x_0$, the wavefront set of the distorted plane wave $u$ is contained in
\begin{align*}
  \WF\l(u|_{M\b \{x_0\}}\r)&\subseteq \Lambda\b\partial\Lambda \\
  &=
  \l\{ \l(\gamma_{x_0,\eta}(s),r\dot{\gamma}_{x_0,\eta}(s)^{\flat} \r) \in T^{\ast}M: \eta\in \W_{x_0,\zeta_0,s_0},s\in \R^{+},r\in \R\b0\r\}.
\end{align*}

To analyze the principal symbol $\sigma[u]$, we follow \cite[Section 2]{CLOP} and consider a bicharacteristic $\beta_{x_0,\zeta_0}(s)$ emanating from  $(x_0,\zeta_0)\in \Sigma\cap \Char(P)$. The bicharacteristic $\beta_{x_0,\zeta_0}(s) = (x(s),\zeta(s))$ satisfies the Hamiltonian system 
\begin{equation}
  \begin{aligned}
    \begin{cases}
    \dot{x}^{i} = g^{ij}(x)\zeta_j,\\
    \dot{\zeta}_i = -\frac{1}{2}\partial_{x^i}g^{jk} \zeta_j\zeta_k,
 \nonumber
    \end{cases}
  \end{aligned}
\end{equation}
with initial data $x(0)=x_0,\zeta(0)=\zeta_0$. 

This trajectory $\beta_{x_0,\zeta_0}(s)$ corresponds to the null geodesic $\gamma_{x_0,\zeta_0}(s)$ via the relation
$$
\beta_{x_0,\zeta_0}(s) = \l(\gamma_{x_0,\zeta_0}(s),\dot\gamma^{\flat}_{x_0,\zeta_0}(s)\r).
$$

We first trivialize the Keller-Maslov line bundle over $\Lambda$ in a conic neighborhood of $\beta_{x_0,\zeta_0}$ before the first cut point. Fix a point $y_0 = \gamma_{x_0,\zeta_0}(\mathbf{t}_0)$ with $\mathbf{t}_0 \in (0,\rho(x_0,\zeta^{\sharp}_0))$. Choose a finite covering of the compact subset $\beta_{x_0,\zeta_0}([0,\mathbf{t}_0])$ of $T^{\ast}M\b 0$. as
$$
  \beta_{x_0,\zeta_0}([0,\mathbf{t}_0]) \subseteq \bigcup_{j=0}^{l} V_{j},
$$
where $V_{0},V_1,\cdots,V_l$ satisfies
\begin{itemize}
 \item $V_j$ is an open conic set in $T^{\ast}M \b 0$ for $j=0,1,\cdots, l$.
 \item $(x_0,\zeta_0)\in V_0$ and $(x_0,\zeta_0)\notin V_k$ for $k=1,\dots,l$. 
 \item $V_k \cap \Lambda$ is a conormal bundle for $k=1,\dots,l$.
\end{itemize}
The last condition is guaranteed by Lemma \ref{lemma: locally conormal distribution}.

The transition functions of Keller-Maslov line bundles in small enough conic neighorhoods are constants, which are powers of $i$,  by \cite[Theorem 3.2.18]{H}, see also the geometric interpretation of the principal symbol in \cite[pp. 14-15]{H4}. Then, the Keller-Maslov line bundle over $V_0\cap \Lambda(x_0,\zeta_0,s_0)$ can be trivialized by $\imath^{m_0}$ for some $m_0\in \Z$. Moreover, the Keller-Maslov line bundle over $V_k\cap \Lambda$ for $k=1,\dots,l$ is trivial by \cite[Theorem 3.3.4]{H}. Then, the Keller-Maslov line bundle over $\cup_{j=0}^{l} V_j$ can be trivialized by $\imath^{m_0}$, which can be neglected without loss of generality.

Fix a strictly positive half-density $\omega$ homogeneous of degree $1/2$ on $\Lambda$. For the existence of such a density, see \cite[Section 2.5]{CLOP}. The divergence of the Hamiltonian vector field $H_P$ with respect to $\omega$ is defined as 
$$
\div_{\omega}H_P := \omega^{-1}\mathscr{L}_{H_P}\omega.
$$
Define a volume factor along the bicharacteristic $\beta(s)$ by 
\begin{equation}\label{eqn: the factor of divergence of Hamilton flow}
  \rho(s) = \int_{0}^{s} \div_{\omega}H_{P }(\beta(s))ds.
\end{equation}
Then, as shown in \cite[(36)]{CLOP}, the principal symbol $\sigma[u]$ of the solution $u$ of \eqref{eqn: linear wave equation with a connection on half-density} with $f\in I^{\mu+3/2}(M;\Sigma;E\otimes \Omega^{1/2})$ satisfies the parallel transport identity 
\begin{equation}\label{eqn: parallel transport of the principal symbol}
 e^{\rho(s)}(\omega^{-1}\sigma[u])(\beta(s)) =  \mathbf{P}^{A}_{\gamma_{x,\xi}([0,s])}(\omega^{-1}\sigma[u])(x,\xi).
\end{equation}
Moreover, the homogeneity of $\sigma[u]$ under rescaling is given in \cite[Proposition 1]{CLOP} as follows, provided that $\sigma[f]$ is positively homogeneous of degree $\mu+5/2$. 
\begin{equation}\label{eqn: homogeneity of the principal symbol}
e^{\rho(s)}(\omega^{-1}\sigma[u])(\beta_{\lambda}(s)) =  \lambda^{\mu+1/2}\mathbf{P}^{A}_{\gamma_{x,\xi}[0,s]}(\omega^{-1}\sigma[u])(x,\xi),\quad \lambda>0,
\end{equation}
where $\beta_{\lambda}(s)= (\gamma(\lambda s),\lambda\dot\gamma^{\flat}(\lambda s))$ is the bicharacteristic emanating from $(x,\lambda\xi)$.

\section{Nonlinear interactions}\label{sec: nonlinear wave interactions}
 In this section, we prove that the source-to-solution map determines the broken light-ray transform uniquely, see Definition \ref{def: broken light ray transform}. 
 \begin{theorem}\label{thm: from source-to-solution map to broken light-ray transforms}
  Let $A$ and $B$ be two Hermitian connections and $A=B$ on $\rotom$. Suppose that the source-to-solution maps coincide, i.e., $L_{g,A} = L_{g,B}$. Then, for all $y\in \D$, $(v,w)\in \E(y)$, and parameters $s^{\prime}\in (0,\rho(y,v))$ and $s^{\prime\prime}\in (0,\rho(y,w))$ such that
  $$
\gamma_{y,v}(s^{\prime})\in \rotom,\quad\text{and}\quad\gamma_{y,w}(s^{\prime\prime})\in \rotom,
  $$
  and
  $$
  (x,\xi) := \l(\gamma_{y,v}(s^{\prime}),-\dot{\gamma}_{y,v}^{\flat}(s^{\prime})\r),
  $$
  we have
  $$
  S^{A}_{\gamma_{x,\xi}([0,s^{\prime}])\gamma_{y,w}([0,s^{\prime\prime}])}=S^{B}_{\gamma_{x,\xi}([0,s^{\prime}])\gamma_{y,w}([0,s^{\prime\prime}])}.
  $$  
 \end{theorem}

  \subsection{Three-fold linearization}
   We only need to prove the case $y\in \D\b\rotom$. The case $y\in\rotom$ is trivial since $A=B$ on $\rotom$. 
   
   By Lemma \ref{lemma: observation points}, there exist $(v,w)\in \E(y)$ and parameters $s^{\prime}\in (0,\rho(y,v))$ and $s^{\prime\prime}\in (0,\rho(y,w))$ such that 
   $$
   \gamma_{y,v}(s^{\prime})\in \rotom,\quad \gamma_{y,w}(s^{\prime\prime})\in \rotom.
   $$

   Choose a local normal coordinate near $y$ such that
   $$
      g|_{y} = -(dx^{0})^{2}+ (dx^{1})^{2}+\cdots+ (dx^{n})^2.
   $$
  By \cite[Lemma 1]{CLOP} and \cite[pp. 388-389]{FO}, we can rotate the coordinates such that the vectors $v$ and $w$ can be written in the following form, for some $\theta\in(0,2\pi)$,
   $$
     v =  (-1, -1, \underbrace{0, \ldots, 0}_{n-1 \text{ times}})\in L_{y}^{-}M,\quad\text{and}\quad w = (1,\cos\theta,\sin\theta,\underbrace{0, \ldots, 0}_{n-2 \text{ times}})\in L_{y}^{+}M.
   $$

We need to choose three linearly independent vectors $w_{(1)},w_{(2)},w_{(3)} \in L_{y}^{-}M$ to construct three causally independent source points. First, let
$$
 w_{(1)} := w.
$$
Next, we perturb the geodesic direction $v_{(1)}$. For sufficiently small $r>0$, define 
   \begin{equation}\label{eqn: perturbation of the vector}
     w_{(2)} = (-1,\sqrt{1-r^2},r,\underbrace{0, \ldots, 0}_{n-2 \text{ times}}),\quad w_{(3)} = (-1,\sqrt{1-r^2},-r,\underbrace{0, \ldots, 0}_{n-2 \text{ times}})\in L_{y}^{-}M.
   \end{equation}
By the lower semi-continuity of the null cut function $\rho(y,\cdot)$ restricted at $y$, there exists a common parameter $s^{\prime}$ such that
$$
x_{(j)}:= \gamma_{y,w_{(j)}}(s^{\prime})\in \rotom,\quad\text{and}\quad s^{\prime}\in (0,\rho(y,w_{(j)})),\quad j=1,2,3.
$$ 
Define
  $$
  \xi_{(j)} = -\dot{\gamma}_{y,w_{(j)}}(s^{\prime}) \in L^{+}_{x_{(j)}}M.
  $$
as the initial vector at $x_{(j)}$ that generates a geodesic joining $x_{(j)}$ to $y$. By choosing $r>0$ sufficiently small, the points $x_{(1)},x_{(2)}$ and $x_{(3)}$ lie on distinct null geodesics and are causally independent, i.e.,
  \begin{equation}\label{eqn: points are causally independent}
    x_{(j)}\notin J^{+}(x_{(k)}),\quad 1\le j\neq k\le3.
  \end{equation}
  Then, the parameter $s^{\prime}$ satisfy 
  $$
    s^{\prime}\in (0,\rho(y,w_{(j)})),\quad\text{for }j=1,2,3. 
  $$
 Moreover, the corresponding reparameterized geodesics satisfy 
  \begin{equation}\label{eqn: reparameter of the geodesics}
   y = \gamma_{x_{(j)},\xi_{(j)}}(s^{\prime}),\quad j=1,2,3.
  \end{equation}

  Define covectors in $L_{y}^{\ast}M$ by
   \begin{equation}\label{eqn: covectors at y}
   \eta := w^{\flat},\quad \eta_{(1)} := -w_{(1)}^{\flat},\quad \eta_{(2)} := w_{(2)}^{\flat},\quad \eta_{(3)} := w_{(3)}^{\flat}.
   \end{equation}
   Then, by \cite[Lemma 1]{CLOP}, there exist $\kappa_{(1)},\kappa_{(2)},\kappa_{(3)}>0$, depending only on $r$ and converging when $r\to 0+$, such that the following linear relation holds.
  \begin{equation}\label{eqn: linear relation}
    \eta = r^{-2}\kappa_{(1)}\eta_{(1)}+ r^{-2}\kappa_{(2)}\eta_{(2)}+ r^{-2}\kappa_{(3)}\eta_{(3)}.
  \end{equation}
 Define the corresponding covectors at $L_{x_{(j)}}M$ by
 \begin{equation}\label{eqn: corresponding covectors}
  \nu_{(1)} := -\xi_{(1)}^{\flat}\in L^{-,\ast}_{x_{(1)}}M, \quad\nu_{(2)} :=\xi_{(2)}^{\flat} \in L^{+,\ast}_{x_{(3)}}M,\quad \nu_{(3)} :=\xi_{(3)}^{\flat} \in L^{+,\ast}_{x_{(3)}}M.
 \end{equation}

Let $\beta_{\out}$ denotes the bicharacertistics emanating from $(y,\eta)$. Then, 
\begin{equation}\label{eqn: bicharacteristic}
  \beta_{\out}(s) = \l(\gamma_{y,w}(s),\dot{\gamma}^{\flat}_{y,w}(s) \r),\quad\beta_{\out}(0) = (y,\eta), \quad \text{and}\quad (z,\zeta) := \beta_{\out}(s^{\prime\prime}).
\end{equation}

 Fix a small parameter $s_0>0$ and $\mu \le -7$. Define the Lagrangian submanifolds corresponding to the source by 
 $$
\Sigma_{(j)} := \Sigma(x_{(j)},\nu_{(j)},s_0),\quad j=1,2,3.
 $$
 We construct point sources of the form
 \begin{equation}\label{eqn: construction of sources}
 f_{(j)} = c_{(j)} \chi_{(j)}\delta_{x_{(j)}},\quad j=1,2,3,
 \end{equation}
 where
  \begin{itemize}
     \item Each $c_{(j)}\in \C^{m}\b0$ is a fixed nonzero vector.
     \item $\chi_{(j)}$ is a pseudodifferential operator whose wavefront set is contained in $\Sigma_{(j)}$ for $j=1,2,3$ and whose Schwartz kernel is scalar valued.
     \item The principal symbol $\sigma[\chi_{(j)}]$ is positively homogeneous of degree $\mu+5/2$.
     \item  $\sigma[\chi_{(j)}] \neq 0$ in a conic neighborhood of $(x_{(j)},\nu_{(j)})$ for $j=1,2,3$.
     \item The sources $f_{(j)}$ are causally independent, i.e.,
      \begin{equation}\label{eqn: sources are causally independent}
     \supp(f_{(j)})\cap J^{+}(\supp(f_{(k)})) = \varnothing,\quad 1\le j \neq k\le 3.
     \end{equation}
  \end{itemize}

  Hence, each source has the following microlocal structure 
  \begin{equation}\label{eqn: three-parameter sources as Lagrangian distributions}
    f_{(j)} \in  I^{\mu+3/2}(\rotom;\Sigma_{(j)};E).
  \end{equation}
 By \cite[Theorem 25.1.4]{H4}, we have the Sobolev regularity for $\mu \le -7$, \begin{equation}\label{eqn: Sobolev regularity}
  f_{(j)}\in H^{4}(\rotom;E).
   \end{equation}

The principal symbol of $f_{(j)}$ at $(x_{(j)},\nu_{(j)})$ is
  \begin{equation}\label{eqn: principal symbol of the three-parameter sources}
 \sigma\l[f_{(j)}\r]\l(x_{(j)},\nu_{(j)} \r) = c_{(j)}\sigma\l[\chi_{(j)}\r]\l(x_{(j)},\nu_{(j)}\r)
  \end{equation}

Let the source of the nonlinear wave equation \eqref{eqn: connection wave equation with a source} be given by
\begin{equation}\label{eqn: three-parameter source of the semilinear wave equation}
  f = \e_{(1)}f_{(1)} +\e_{(2)}f_{(2)} +\e_{(3)}f_{(3)}.
\end{equation}
The solution $v(\epsilon)$ of \eqref{eqn: connection wave equation with a source}, corresponding to the source \eqref{eqn: three-parameter source of the semilinear wave equation}, depends smoothly on the parameter $\epsilon = (\e_{(1)},\e_{(2)},\e_{(3)})$, see \cite[Section 6.5]{Rauch-book}. 

Define half-density valued function, see \eqref{eqn: correspond to half-density valued function}, by
$$
u(\epsilon) := v(\epsilon)|g|^{1/4}
$$ 
which solves the conjugated equation \eqref{eqn: connection wave equation of half-density valued} with the source $|g|^{1/4}f$.

For $j=1,2,3$, define the first order derivative at $\epsilon=0$ as
$$
u_{(j)} = \partial_{\e_{(j)}}u(\epsilon)|_{\epsilon = 0}.
$$
Differentiating \eqref{eqn: connection wave equation of half-density valued} with respect to $\e_{(j)}$ at $\epsilon= 0$ yields the one-fold linearized equation
\begin{equation}\label{eqn: one-fold linearized equation}
 \begin{aligned}
 \begin{cases}
  Pu_{(j)} = |g|^{1/4}f_{(j)},\quad &\text{on } (0,T)\times N,\\
  u_{(j)}(0,x^{\prime})=0,\partial_{t}u_{(j)}(0,x^{\prime})=0 \quad  &x^{\prime}\in N.
  \end{cases}
  \end{aligned}
\end{equation}

Moreover, define the three-fold linearized wave $u_{(123)}  = \partial_{\e_{(1)}}\partial_{\e_{(2)}}\partial_{\e_{(3)}}u(\epsilon)|_{\epsilon = 0}$ solving
\begin{equation}\label{eqn: three-fold linearized equation}
\begin{aligned}
 \begin{cases}
   Pu_{(123)}  = f_{(123)},\quad &\text{on }(0,T)\times N,\\
   u_{(123)}(0,x^{\prime})=0, \partial_{t}u_{(123)}(0,x^{\prime})=0,\quad &x^{\prime}\in N.
   \end{cases}
   \end{aligned}
\end{equation}
where
$$
f_{(123)}:=\sum_{\tau\in S(3)}\frac{1}{2} \operatorname{Re}\l(\l\langle |g|^{-1/4}u_{(\tau(1))},|g|^{-1/4}u_{(\tau(2))}\r\rangle\r)u_{(\tau(3))}.
$$
Moreover, $S(3)$ denotes the permutation group on $\{1,2,3\}$ and  $\langle\cdot,\cdot\rangle$ is the bundle metric on $E$. The real part is taken in $\C$.

Let $\Lambda_{(j)}$ denote the future flowout of $\Sigma_{(j)} \cap \Char(P)$ under the Hamiltonian flow. By the microlocal structure of $f_{(j)}$ \eqref{eqn: three-parameter sources as Lagrangian distributions}, the linearized equation \eqref{eqn: one-fold linearized equation}, and Theorem \ref{thm: paramatrix of the wave equations}, the solution $u_{(j)}$ is an intersecting paired Lagrangian distribution, i.e.,
$$
  u_{(j)} \in I^{\mu}(M;\Sigma_{(j)},\Lambda_{(j)};E\otimes\Omega^{1/2}).
$$

\subsection{Singularity analysis}
In this section, we analyze the singularities of the third order linearized wave $u_{(123)}$, which solves the equation \eqref{eqn: three-fold linearized equation}.

By Lemma \ref{lemma: locally conormal distribution}, see also the reparametrization of the geodesics, each first order linearzed wave $u_{(j)}$ is locally a conormal distribution near $y$. That is, there exists a neighborhood $V_0\subseteq M$ of $y$ such that
\begin{equation}\label{eqn: u_j are locally conormal distributions}
  u_{(j)}|_{V_0} \in I^{\mu}(N^{\ast}K_{(j)};E\otimes\Omega^{1/2}),
\end{equation}
where each $K_{(j)}\cap V_0$ is a smooth submanifold of codimension $1$.

Moreover, for the covectors $\eta_{(j)}\in L_{y}^{\ast}M$ defined in \eqref{eqn: covectors at y}, we have 
\begin{equation}\label{eqn: wavefront set of the one-fold linearization near y}
 \l(y,\eta_{(j)} \r) \in N^{\ast}K_{(j)},\quad 1\le j\le 3.
\end{equation}

Since the hypersurfaces $K_{(j)}\cap V_0$ intersects transversally near $y$, their conormal bundles satisfy
\begin{equation}\label{eqn: conormal bundle of transversal intersection I}
  N^{\ast}_{x}\l(K_{(i)}\cap K_{(j)} \r) = N^{\ast}_{x}K_{(i)}\oplus N^{\ast}_{x}K_{(j)}, \quad x\in V_0\cap K_{(i)} \cap K_{(j)},\quad 1\le i<j\le 3.
\end{equation}
Moreover, for any permutation $\tau\in S(3)$, we have that $K_{(\tau(1))}\cap K_{(\tau(2))}$ intersects $K_{(\tau(3))}$ transversally. Then, 
\begin{equation}\label{eqn: conormal bundle of transversal intersection II}
 N^{\ast}_{x}\l(K_{(1)}\cap K_{(2)}\cap K_{(3)} \r) = \oplus_{j=1}^{3}N^{\ast}_{x}K_{(j)},\quad x\in V_0\cap \l(\cap_{j=1}^{3} K_{(j)}\r).
\end{equation}

Combining the transversal intersection properties in \eqref{eqn: conormal bundle of transversal intersection I} and \eqref{eqn: conormal bundle of transversal intersection II}, together with the wavefront set calculus for pointwise products of distributions \cite[Theorem 8.2.10]{H1}, we obtain that
$$
\WF\l(f_{(123)}|_{V_0} \r)\subseteq \bigcup_{j=1}^{3} N^{\ast}K_{(j)} \cup \bigcup_{1\le j<k\le 3} N^{\ast}\l( K_{(j)}\cap K_{(k)}\r) \cup  N^{\ast}\l( K_{(1)}\cap K_{(2)}\cap K_{(3)}\r).
$$
Moreover, by the linear relation in \eqref{eqn: linear relation}, the wavefront set of $u_{(j)}$ in \eqref{eqn: wavefront set of the one-fold linearization near y}, and the direct sum decomposition \eqref{eqn: conormal bundle of transversal intersection II}, we conclude that
\begin{equation}\label{eqn: conormal bundle associated with (y,eta)}
  (y,\eta) \in N^{\ast}\l(K_{(1)}\cap K_{(2)}\cap K_{(3)}\r).
\end{equation}

\begin{lemma}\label{lemma: property of (y,eta)}
  For all $1\le j< k\le 3$, we have
  $$
  (y,\eta)\notin  N^{\ast}K_{(j)},\quad\text{and}\quad (y,\eta)\notin N^{\ast}\l(K_{(j)}\cap K_{(k)} \r).
  $$
\end{lemma}
\begin{proof}
 Assume that $(y,\eta)\in N^{\ast}K_{(j)}$ for some $j=1,2,3$. Then, there exists a light-like geodesic $\gamma_{x_{(j)},\nu}$ joining $x_{(j)}$ to $y$. Then, we can write
 $$
   y= \gamma_{x_{(j)},\nu}(\tilde{s}),\quad\text{and}\quad\eta =\dot{\gamma}_{x_{(j)},\nu}^{\flat}(\tilde{s})\in N^{\ast}_{y}K_{(j)}\b0.
 $$
 
 Recall that $\gamma_{x_{(j)},\xi_{(j)}}$ is a geodesic from $x_{(j)}$ to $y$ such that $\gamma_{x_{(j)},\xi_{(j)}}(s_{j}^{\prime}) = y$ and $\dot{\gamma}^{\flat}_{x_{(j)},\xi_{(j)}}(s_{j}^{\prime}) = -\eta_{(j)}\in N^{\ast}_{y}K_{(j)}\b 0$. Since $\eta\neq \eta_{(j)}$ by the linear relation \eqref{eqn: linear relation}, we conclude that $\gamma_{x_{(j)},\nu}$ and $\gamma_{x_{(j)},\xi_{(j)}}$ are two light-like geodesics joining $x_{(j)}$ to $y$, which contradicts the fact that there are no cut points along $\gamma_{x_{(j)},\xi_{(j)}}$ from $x_{(j)}$ to $y$ by Lemma \ref{lemma: only causal curve}. Then, $(y,\eta)\notin N^{\ast}K_{(j)}$ for all $1\le j\le 3$.

Moreover, the sum of two linearly independent light-like covectors is not light-like by \cite[Lemma 27, pp. 141]{O}. By the direct sum \eqref{eqn: conormal bundle of transversal intersection I}, we have
$$
 (y,\eta) \notin N^{\ast}\l(K_{(j)}\cap K_{(k)} \r)\b\l(N^{\ast}K_{(j)}\cup N^{\ast}K_{(k)} \r).
$$
Then, we have finished the proof.
\end{proof}

By Lemma \ref{lemma: property of (y,eta)} and \eqref{eqn: conormal bundle associated with (y,eta)}, we can choose a pseudodifferential operator $\chi$ whose wavefront set is  contained in a small conic neighborhood of $(y,\eta)$, and satisfies 
\begin{equation}\label{eqn: property of the microlocal cut-off}
\sigma[\chi](y,\eta) =1,\quad\WF(\chi) \cap \l(\bigcup_{j=1}^{3} N^{\ast}K_{(j)} \cup \bigcup_{1\le k<l\le 3}N^{\ast}\l(K_{(k)} \cap K_{(l)}\r)\r) = \varnothing.
\end{equation}
Then we have
$$
\WF\l(\chi f_{(123)}\r)  \subseteq N^{\ast}\l( K_{(1)}\cap K_{(2)}\cap K_{(3)}\r).
$$

We now show that $(\id-\chi)f_{(123)}$ does not propagate singularities to $(z,\zeta)$.

By \cite[Proposition 1.3.1]{D}, the singular support of a distribution coincides with the projection of its wavefront set onto the base manifold. Since 
$$
 \singsupp\l( u_{(j)}\r)\b\{x_{(j)}\} \subseteq   K\l(x_{(j)},\xi_{(j)},s_0 \r),\quad j=1,2,3,
$$
where $K(\cdot,\cdot,\cdot)$ is defined in \eqref{eqn: flowout submanifold}. Hence, 
\begin{equation}\label{eqn: singular support of the source}
  \singsupp\l(f_{(123)}\r)\b\bigcup_{j=1}^{3}\{x_{(j)}\} \subseteq \bigcup_{j=1}^{3}  K\l(x_{(j)},\xi_{(j)},s_0 \r).
\end{equation}

Recall the bicharacteristic $\beta_{\out}$ in \eqref{eqn: bicharacteristic}. We reparameterize it as
\begin{equation}\label{eqn: reparameterized bicharacteristic}
  \beta_{z,-\zeta^{\sharp}}(s) = \l(\gamma_{z,-\zeta^{\sharp}}(s),\dot{\gamma}^{\flat}_{z,-\zeta^{\sharp}}(s) \r),\quad s\in [0,\mathcal{T}(z,-\zeta^{\sharp})).
\end{equation}
It suffices to prove the following result.

\begin{lemma}\label{lemma: propagation of regularities}
    There exists $s_0>0$ such that, for all $j=1,2,3$, the flowout $K(x_{(j)},\xi_{(j)},s_0)$ intersects the geodesic segment $\gamma_{z,-\zeta^{\sharp}}[0,\mathcal{T}(z,-\zeta^{\sharp}))$ only at the point $y$.
\end{lemma}

\begin{proof}
 Let $(x,\xi):=(x_{(j)},\xi_{(j)})$ and denote $K(x,\xi,s_0):= K(x_{(j)},\xi_{(j)},s_0)$ for simplicity. Assume that for all $s_0>0$, there exists a point $\tilde{y}_{s_0} = \gamma_{z,-\zeta^{\sharp}}(\tilde{s}) \in K(x,\xi,s_0)$ with $\tilde{y}_{s_0}\neq y$. By the definition of $K(x,\xi,s_0)$ in \eqref{eqn: flowout submanifold}, this means there exists a future-pointing null geodesic $\gamma_{x,\nu_{s_0}}$ from $x$ to $\tilde{y}_{s_0}$, with $\nu_{s_0}\in \W_{x,\xi,s_0}\subseteq LM$.

 Since both $y$ and $\tilde{y}_{s_0}$ lie on the same geodesic  null segment $\gamma_{z,-\zeta^{\sharp}}([0,\mathcal{T}(z,-\zeta^{\sharp})))$, we divide the proof into two cases.

  The first case is $\tilde{y}_{s_0} <y$. Then, the geodesic segment $\gamma_{x,\xi}$ from $x$ to $y$ and the broken piecewise smooth causal curve $\tilde{\gamma}$ formed by concatenating $\gamma_{x,\nu_{s_0}}$ from $x$ to $\tilde{y}_{s_0}$ and $\gamma_{z,-\zeta^{\sharp}}$ from $\tilde{y}_{s_0}$ to $y$, are two distinct causal paths from $x$ to $y$. This contradicts the fact that $\gamma_{x,\xi}([0,s^{\prime}]),s^{\prime}\in(0,\rho(x,\xi))$ is the unique causal curve from $x$ to $y$ according to the Lemma \ref{lemma: only causal curve}.

  The second case is $y<\tilde{y}_{s_0}\le z$. Then there are two distinct future-pointing piecewise smooth causal curves from $x$ to $\tilde{y}_{s_0}$. One is the lightlike geodesic $\gamma_{x,\nu_{s_0}}$, and the other is the broken null geodesic path consisting of the segments $\gamma_{x,\xi}|_{[0,s^{\prime}]}$ and $\gamma_{y,w}$ from $y$ to $\tilde{y}_{s_0}$. This implies that $\tilde{y}_{s_0}$ lies beyond the first cut point along $\gamma_{x,\nu_{s_0}}$ by Lemma \ref{lemma: only causal curve}.

  Recall that $y=\gamma_{x,\xi}(s^{\prime})$ for some $s^{\prime}\in (0,\rho(x,\xi))$. By the lower semi-continuity of the null cut function $\rho$, for any $\e>0$, there exists $\delta>0$ such that for all $\nu\in \W_{x,\xi,s_0}$ with $s_0<\delta$,
  $$
     \rho(x,\nu) > \rho(x,\xi)-\e >s^{\prime}.
  $$
  Fix such $\e$ and $\delta$. Then, there exists a sequence $\tilde{y}_{k} = \gamma_{x,\nu_{k}}(\tilde{s}_k)\in K(x,\xi,s_0)$ such that $\nu_k\to\xi$ and $\tilde{s}_k >\rho(x,\xi)-\e$. Since each $\tilde{y}_{k}$ lies on the compact geodesic segment $\gamma_{y,w}|_{[0,s^{\prime\prime}]}$,
  there exists a convergent subsequence. Without loss of generality, assume $\tilde{y}_{k}\to \tilde{y}$.  By \cite[Lemma 6.1]{FLO}, the corresponding parameters $\tilde{s}_{k}\to \tilde{s}\ge \rho(x,\xi)-\e$. Then, by \cite[Lemma 9.34]{Beem-Ehrlich-Easley}, we conclude that
$$
   \tilde{y}_{k} = \gamma_{x,\nu_k}(\tilde{s}_{k}) \to \tilde{y} = \gamma_{x,\xi}(\tilde{s})\in \gamma_{y,w}[0,s^{\prime\prime}],\quad \tilde{s}\ge \rho(x,\xi)-\e > s^{\prime}.
$$
This implies that both $y=\gamma_{x,\xi}(s^{\prime})$ and $\tilde{y} = \gamma_{x,\xi}(\tilde{s})$ lie on the same geodesic $\gamma_{x,\xi}$, and are connected to each other via the segment $\gamma_{y,w}$. This contradicts the fact that there are no cut points along $\gamma_{y,w}|_{[0,s^{\prime\prime}]}$.
\end{proof}

Note that the projection of $\beta_{z,-\zeta^{\sharp}}$ onto $M$ is the geodesic $\gamma_{z,-\zeta^{\sharp}}$. Since the manifold is globally hyperbolic, there are no closed null geodesics. Then, Lemma \ref{lemma: propagation of regularities}, together with the singular support of $f_{(123)}$ in \eqref{eqn: singular support of the source}, implies that
\begin{equation}\label{eqn: propagation of regularity}
  \{\beta_{z,-\zeta^{\sharp}}(s): s\in [0,\mathcal{T}(z,-\zeta^{\sharp}))\} \cap \WF\l(f_{(123)} \r) = \{(y,\eta)\}.
\end{equation}

Recall the pseudodifferential operator $\chi$ constructed in \eqref{eqn: property of the microlocal cut-off}. Decompose the solution $u_{(123)}$ of the third order linearized equation \eqref{eqn: three-fold linearized equation} as
\begin{equation}\label{eqn: decompose u_123}
u_{(123)} = v_{(123)} + w_{(123)},
\end{equation}
where $v_{(123)}$ and $w_{(123)}$ respectively solve the equations
\begin{equation} \label{Microlocal cut-off equation}
\begin{aligned}
  \begin{cases}
  Pv_{(123)} = \chi f_{(123)},\quad &\text{on }(0,T)\times N,\\
  v_{(123)}(0,x^{\prime}) = 0,\partial_{t}v_{(123)}(0,x^{\prime})=0,\quad & x^{\prime}\in N,
  \end{cases}
\end{aligned}
\end{equation}
and
\begin{equation}
\begin{aligned}
  \begin{cases}
  Pw_{(123)} =(\id-\chi)f_{(123)},\quad &\text{on }(0,T)\times N,\\
  w_{(123)}(0,x^{\prime}) = 0,\partial_{t}w_{(123)}(0,x^{\prime})=0,\quad & x^{\prime}\in N.
  \end{cases}
\end{aligned}
\end{equation}

By \eqref{eqn: propagation of regularity}, the bicharacteristic segment $\beta_{z,-\zeta}|_{[0,\mathcal{T}(z,-\zeta^{\sharp}))}$ defined by \eqref{eqn: reparameterized bicharacteristic} does not intersect the wavefront set of $(\id-\chi)f_{(123)}$. Then, by the propagation of singularities theorem \cite[Theorem 6.1.1]{DH}, together with the fact that $w_{(123)}$ has vanishing initial data, we obtain
\begin{equation}\label{eqn: regularity of w_123}
(z,\zeta)\notin \WF\l(w_{(123)}\r)
\end{equation}

To compute the pincipal symbol of $\chi f_{(123)}$, we apply the Greenleaf-Uhlmann calculus for pointwise products of conormal distributions \cite[Lemma 1.1]{GU}, see also \cite[Lemma 3.3, 3.6]{LUW}. First, we define the product of half-densities on the conormal bundles of submanifolds intersecting transversally.
\begin{definition}\label{def: definition of the product of the half-density}
  Let $Y_{(1)}$ and $Y_{(2)}$ be smooth submanifolds of the $n$-dimensional smooth manifold $X$ intersecting transversally. In local slice coordinates, $Y_{(1)}$ and $Y_{(2)}$ are expressed as
  \begin{equation}\label{eqn: local slice coordinate}
  \begin{aligned}
   Y_{(1)} &=\{(x^{\prime},x^{\prime\prime},x^{\prime\prime\prime})\in \R^{d_1}\times\R^{d_2}\times\R^{n-d_1-d_2}\mid x^{\prime}=0\},\\
   Y_{(2)} &=\{(x^{\prime},x^{\prime\prime},x^{\prime\prime\prime})\in \R^{d_1}\times\R^{d_2}\times\R^{n-d_1-d_2}\mid x^{\prime\prime}=0\}.
  \end{aligned}
  \end{equation}
   These coordinates exist, see for example \cite[Appendix C]{H3}. Let  $\omega_{(j)}$ be half-densities on $N^{\ast}Y_{(j)}$ for $j=1,2$, which in local slices coordinate takes expression as
\begin{align*}
  \omega_{(1)}(x^{\prime\prime},x^{\prime\prime\prime},\xi^{\prime}) = a_{(1)}(x^{\prime\prime},x^{\prime\prime\prime},\xi^{\prime})\l|dx^{\prime\prime}\wedge dx^{\prime\prime\prime} \r|^{1/2}|d\xi^{\prime}|^{1/2},\\
   \omega_{(2)}(x^{\prime},x^{\prime\prime\prime},\xi^{\prime\prime}) = a_{(2)}(x^{\prime},x^{\prime\prime\prime},\xi^{\prime\prime})\l|dx^{\prime}\wedge dx^{\prime\prime\prime} \r|^{1/2}|d\xi^{\prime\prime}|^{1/2},
\end{align*}
where $a_{(j)}$ is a smooth complex valued function on $N^{\ast}Y_{(j)}$. For any $\mu$, a half-density on $X$, we define the product on $N^{\ast}\l(Y_{(1)}\cap Y_{(2)}\r)$ by
\begin{equation}\label{eqn: half-density for transversal intersection}
 \l(\omega_{(1)}\omega_{(2)}\r)(x^{\prime\prime\prime},\xi^{\prime},\xi^{\prime\prime}) = \mu^{-1}(0,0,x^{\prime\prime\prime})\omega_{(1)}(0,x^{\prime\prime\prime},\xi^{\prime} )\omega_{(2)}(0,x^{\prime\prime\prime},\xi^{\prime\prime}).
\end{equation}
\end{definition}

Next, we prove that the product of half-densities restricted on the cornormal bundle of the intersecting submanifold defined in \eqref{eqn: half-density for transversal intersection} is a well-defined half-density.
\begin{lemma}\label{lemma: well-defined half-density}
 The product \eqref{eqn: half-density for transversal intersection} is a half-density on $N^{\ast}\l(Y_{(1)}\cap Y_{(2)} \r)$.
\end{lemma}
\begin{proof}
 Suppose that $(\tilde{x}^{\prime},\tilde{x}^{\prime\prime},\tilde{x}^{\prime\prime\prime})$ is a system of coordinates such that $Y_{(1)}$ and $Y_{(2)}$ are of the form \eqref{eqn: local slice coordinate} in these coordinates, that is, $\tilde{x}^{\prime}(0,x^{\prime\prime},x^{\prime\prime\prime})=0$ and $\tilde{x}^{\prime\prime}(x^{\prime},0,x^{\prime\prime\prime})=0$. It induces a coordinate change in $T^{\ast}X$ given by $(\tilde{x},\tilde{\xi})$. Then,
 \[ \l.\frac{d\tilde{x}}{dx}\r|_{x^{\prime\prime\prime}=0}=
\begin{pmatrix}
\frac{d\tilde{x}^{\prime}}{dx^{\prime}} & 0 & 0 \\
0 & \frac{d \tilde{x}^{\prime\prime}}{dx^{\prime\prime}} &0 \\
\ast &\ast &\frac{d\tilde{x}^{\prime\prime\prime}}{dx^{\prime\prime\prime}}
\end{pmatrix}.
\]
Then, under the change of coordinates,
\begin{align*}
 &\mu^{-1}(0,0,x^{\prime\prime\prime})\omega_{(1)}(0,x^{\prime\prime\prime},\xi^{\prime})\omega_{(2)}(0,x^{\prime\prime\prime},\xi^{\prime\prime})\\
 =& \l|\frac{d\tilde{x}}{dx} \r|^{-1/2} \mu^{-1}(0,0,\tilde{x}^{\prime\prime\prime}) \l|\frac{d(\tilde{x}^{\prime\prime},\tilde{x}^{\prime\prime\prime})}{d(x^{\prime\prime},x^{\prime\prime\prime})} \r|^{1/2} \l|\frac{d x^{\prime}}{d\tilde{x}^{\prime}} \r|^{1/2}\omega_{(1)}(0,\tilde{x}^{\prime\prime\prime},\tilde{\xi}^{\prime})\\
 &\cdot \l|\frac{d(\tilde{x}^{\prime},\tilde{x}^{\prime\prime})}{d(x^{\prime},x^{\prime\prime})} \r|^{1/2}  \l|\frac{d x^{\prime\prime}}{d\tilde{x}^{\prime\prime}} \r|^{1/2} \omega_{(2)}(0,\tilde{x}^{\prime\prime\prime},\tilde{\xi}^{\prime\prime})\\
 =& \l|\frac{d \tilde{x}^{\prime\prime\prime}}{d x^{\prime\prime\prime}} \r|^{1/2} \l|\frac{d(x^{\prime},x^{\prime\prime})}{d(\tilde{x}^{\prime},\tilde{x}^{\prime\prime})} \r|^{1/2} \mu^{-1}(0,0,\tilde{x}^{\prime\prime\prime}) \omega_{(1)}(0,\tilde{x}^{\prime\prime\prime},\tilde{\xi}^{\prime})\omega_{(2)}(0,\tilde{x}^{\prime\prime\prime},\tilde{\xi}^{\prime\prime}),
\end{align*}
which guarantees that the product \eqref{eqn: half-density for transversal intersection} restricted on $N^{\ast}\l(Y_{(1)}\cap Y_{(2)}\r)$ is a well-defined half-density.
\end{proof}

Observe that 
$$
N^{\ast}\l(Y_{(1)} \cap Y_{(2)}\r) = \l\{ (x^{\prime},x^{\prime\prime},x^{\prime\prime\prime}; \xi^{\prime},\xi^{\prime\prime},\xi^{\prime\prime\prime})\mid x^{\prime } = 0, x^{\prime\prime}=0, \xi^{\prime\prime\prime}=0\r\},
$$
and for $p \in Y_{(1)}\cap Y_{(2)}$, we have that
$$
 N^{\ast}_{p}\l(Y_{(1)} \cap Y_{(2)}\r) = N^{\ast}_{p}Y_{(1)}\oplus N^{\ast}_{p}Y_{(2)}.
$$
Moreover, if $\xi\in N^{\ast}_{p}\l(Y_{(1)}\cap Y_{(2)} \r)$, then there is a unique decomposition
$$
 \xi = \xi_{(1)}+\xi_{(2)},\quad \xi_{(j)}\in N^{\ast}_{p}Y_{(j)}, \text{ for } j=1,2.
$$
In local coordinates, 
$$
p=(0,0,x^{\prime\prime\prime}),\quad \xi = (\xi^{\prime},\xi^{\prime\prime},0),\quad
\xi_{(1)}=(\xi^{\prime},0,0),\quad \xi_{(2)} = (0,\xi^{\prime\prime},0),
$$
for some $x^{\prime\prime\prime}\in \R^{n-d_1-d_2},\xi^{\prime}\in\R^{d_1}$ and $\xi^{\prime\prime}\in \R^{d_2}$. We write
\begin{equation}\label{eqn: product of half-density generally}
\mu^{-1}(p)\omega_{(1)}\l(p,\xi_{(1)}\r)\omega_{(2)}\l(p,\xi_{(2)}\r) := \mu^{-1}(0,0,x^{\prime\prime\prime})\omega_{(1)}(0,x^{\prime\prime\prime},\xi^{\prime})\omega_{(2)}(0,x^{\prime\prime\prime},\xi^{\prime\prime}).
\end{equation}

The following proposition is Greenleaf-Uhlmann calculus for pointwise products of conormal distributions \cite[Lemma 1.1]{GU}.
\begin{proposition}\label{prop: Greenlead-Uhlmann's calculus}
  Let $Y_{(1)}$ and $Y_{(2)}$ be two smooth submanifolds of a smooth manifold $X$ intersecting transversally, with local slice coordinates of the form \eqref{eqn: local slice coordinate}. Denote their codimension by $d_1$ and $d_2$, respectively. Suppose $v_{(j)}\in I^{\mu_j}(N^{\ast}Y_{(j)};E\otimes\Omega^{1/2})$ for $j=1,2$. 
  
  The pointwise product $v_{(1)}v_{(2)}$ is defined via
  \begin{equation}\label{eqn: pointwise product of distributions}
    v_{(1)}v_{(2)} = \l(\mu^{-1}v_{(1)} \r)\l(\mu^{-1}v_{(2)} \r)\mu,
  \end{equation}
  where $\mu$ is a fixed strictly positive half-density on $X$.

  Let $\chi$ be a pseudodifferential operator such that $\WF(\chi) \cap N^{\ast}Y_{(j)} = \varnothing$ for $j=1,2$. Then,
  $$
   \chi\l(v_{(1)}v_{(2)} \r)\in I^{\mu_1+\mu_2+\frac{n}{4}}\l(N^{\ast}\l(K_{(1)}\cap K_{(2)} \r); E\otimes\Omega^{1/2}\r),
  $$
  with principal symbol given by
  $$
   \sigma\l[\chi\l(v_{(1)}v_{(2)} \r)\r](x,\xi) = C(d_1,d_2)\sigma[\chi](x,\xi)\mu^{-1}(x)\sigma\l[v_{(1)}\r]\l(x,\xi_{(1)}\r)\sigma\l[v_{(2)}\r]\l(x,\xi_{(2)}\r),
  $$
  where $\xi = \xi_{(1)}+\xi_{(2)}$ with $(x,\xi_{(j)})\in N^{\ast}K_{(j)}\b 0$, and $C(d_1,d_2)\in \R^{+}$ is a constant dependent only on the codimensions $d_1$ and $d_2$, and the product is defined by \eqref{eqn: product of half-density generally}.
\end{proposition}

For any permutation $(i,j,k)$ of $(1,2,3)$, choose a microlocal cut-off $\chi_0$ in a conic neighborhood of $(y,r^{-2}\kappa_{(i)}\eta_{(i)}+r^{-2}\kappa_{(j)}\eta_{(j)})$. Although $\l\langle |g|^{-1/4}u_{(i)}, |g|^{-1/4}u_{(j)}\r\rangle$ is not a conormal distribution, $\chi_0\l( \l\langle |g|^{-1/4}u_{(i)}, |g|^{-1/4}u_{(j)}\r\rangle\r)$ is. By the proof of \cite[(56)]{CLOP}, see also the proof of \cite[Lemma 4.8]{KLOU}, 
\begin{multline}
\chi \l( \mathrm{Re}\l(\l\langle |g|^{-1/4}u_{(i)}, |g|^{-1/4}u_{(j)}\r)\r\rangle u_{(k)}\r)  \\
=\chi \l(\chi_0 \l( \mathrm{Re}\l(\l\langle |g|^{-1/4}u_{(i)}, |g|^{-1/4}u_{(j)}\r)\r\rangle\r)u_{(k)}\r) \mod C^{\infty}.
\end{multline}
Then, by applying Proposition \ref{prop: Greenlead-Uhlmann's calculus} twice, we obtain 
\begin{equation}\label{eqn: microlocal structure of the interaction source}
\chi f_{(123)} \in I^{3\mu+2}\l(N^{\ast}(\cap_{j=1}^{3}K_{(j)});E\otimes\Omega^{1/2}\r),
\end{equation}
and the principal symbol at $(y,\eta)$ is given by
\begin{multline}\label{eqn: principal symbol of the interaction source}
 \sigma\l[\chi f_{(123)}\r](y,\eta)=\frac{1}{2}|g(y)|^{-1/2}\sigma\l[u_{(\tau(3))}\r]\l(y,r^{-2}\kappa_{(\tau(3))}\eta_{(\tau(3))}\r)\\
 \times\sum_{\tau\in S(3)}\operatorname{Re}\l\langle \sigma\l[u_{(\tau(1))}\r]\l(y,r^{-2}\kappa_{(\tau(1))}\eta_{(\tau(1))}\r),\sigma\l[u_{(\tau(2))}\r]\l(y,r^{-2}\kappa_{(\tau(2))}\eta_{(\tau(2))}\r)\r\rangle, 
\end{multline}
where the inner product is taken with respect to a local trivialization of $E\otimes\Omega_{1/2}$, and is locally the standard Hermitian inner product on $\C^{n}$. Here, $\Omega_{1/2}$ denotes the half-density bundle over $N^{\ast}(\cap_{j=1}^{3}K_{(k)})$.

\subsection{Proof of the main theorem}
In this section, we compute an explicit expression of and  principal symbol of $v_{(123)}$ to recover the broken light-ray transform from the source-to-solution map, thereby completing the proof of Theorem \ref{thm: main theorem}. Our argument is inspired by the proof of \cite[Theorem 4]{CLOP}

\begin{proof}[Proof of Theorem \ref{thm: from source-to-solution map to broken light-ray transforms}]

Recall that $u_{(1)},u_{(2)}$ and $u_{(3)}$ are conormal distributions near $y$, as given in \eqref{eqn: u_j are locally conormal distributions}. Also recall the null geodesics from $x_{(1)},x_{(2)}$ and $x_{(3)}$ to $y$ are parameterized so that \eqref{eqn: reparameter of the geodesics} holds with corresponding covectors defined in \eqref{eqn: corresponding covectors}. For $k=2,3$, let $\beta_{(k)}(s)$ denotes the bicharacteristic such that
$$
\beta_{(k)}(0) =\l(x_{(k)},\nu_{(k)}\r),\quad\beta_{(k)}(s^{\prime}) = \l(y,\eta_{(k)}\r),\quad \nu_{(k)}\in L_{x_{(j)}}^{+,\ast}M.
$$ 
Similarly, define $\beta_{(1)}$ to be the bicharacteristic such that
$$
\beta_{(1)}(0) =\l(x_{(1)},\nu_{(1)}\r),\quad\beta_{(1)}(-s^{\prime}) = \l(y,\eta_{(1)}\r),\quad \nu_{(1)}\in L_{x_{(1)}}^{-,\ast}M.
$$
The principal symbol of $u_{(j)}$ for $j=1,2,3$ on the bicharacteristic $\beta_{(j)}(s)$ satisfies the transport equation \eqref{eqn: transport equation for principal symbol}, with initial value
\begin{align*}
  \sigma\l[u_{(j)}\r]\l(x_{(j)},\zeta_{(j)}\r) &= \mathscr{R}\l(\sigma[\Box_g]^{-1}\sigma\l[|g|^{-1/4}f_{(j)}\r] \r)\l(x_{(j)},\nu_{(j)}\r)\\
  &= c_{(j)}|g(x_{(j)})|^{-1/4}\mathscr{R}\l( \sigma[\Box_g]^{-1}\sigma\l[\chi_{(j)}\r]\r)\l(x_{(j)},\nu_{(j)}\r)\\
  &:= c_{(j)}\alpha_{(j)},
\end{align*}
where $\alpha_{(j)}$ is a nonvanishing complex-valued factor. We may write
$$
 \sigma\l[u_{(j)}\r]\l(x_{(j)},\zeta_{(j)} \r) = \tilde{c}_{(j)}\tilde{\alpha}_{(j)},\quad\tilde{\alpha}_{(j)}\in \R^{+},
$$ 
where
$$
 \quad \tilde{c}_{(j)} = c_{(j)} e^{\imath\arg\l(\alpha_{(j)}\r)},\quad\text{and}\quad \tilde{\alpha}_{(j)} = \l| \alpha_{(j)}\r|.
$$
Here $|\cdot|$ denotes the modulus and $\arg(\cdot)$ denotes the argument in $\C$. 

In the construction of the source \eqref{eqn: construction of sources}, we may choose the constants $c_{(j)}\in \C^{n}\b0$ so that the corresponding vectors $\tilde{c}_{(j)}$ satisfy the following identity
\begin{equation}\label{eqn: equality of tilde c_j by construction}
\tilde{c}_{(1)} = \tilde{c}_{(2)} = \tilde{c}_{(3)},\quad\text{and}\quad \l|\tilde{c}_{(j)}\r| =1.
\end{equation}

For each $j=1,2,3$, denote by $\omega_{(j)}$ the half-density on $N^{\ast}K_{(j)}\b0$ , which is a strictly positive function on $N^{\ast}K_{(j)}\b0$, homogeneous of degree $1/2$. Denote the volume factor as in \eqref{eqn: the factor of divergence of Hamilton flow} by
$$
\rho_{(j)}(s) = \int_{0}^{s}\div_{\omega}H_{P}\l(\beta_{(j)}(\tilde{s})\r)d\tilde{s}\in\R.
$$

For $k=2,3$, by the homogeneity of the principal symbol given in \eqref{eqn: homogeneity of the principal symbol}, we have 
\begin{equation}\label{eqn: principal symbol at y: I}
\l(\omega_{(k)}^{-1}\sigma\l[u_{(k)}\r]\r)\l(y,r^{-2}\kappa_{(k)}\eta_{(k)}\r)  = \l(r^{-2}\kappa_{(k)}\r)^{\mu+1/2} e^{-\rho_{(k)}(s^{\prime})} \tilde{\alpha}_{(k)} \mathbf{P}^{A}_{\gamma_{x_{(k)},\xi_{(k)}}([0,s^{\prime}])}\tilde{c}_{(k)}.
\end{equation}

Moreover, we compute the principal symbol of $u_{(1)}$ as follows
\begin{align}
\l(\omega_{(1)}^{-1}\sigma\l[u_{(1)}\r]\r)\l(y,r^{-2}\kappa_{(1)}\eta_{(1)}\r)
 \notag & = \l(r^{-2}\kappa_{(1)}\r)^{\mu+1/2} e^{-\rho_{(1)}(-s^{\prime})} \tilde{\alpha}_{(1)} \mathbf{P}^{A}_{\gamma_{x_{(1)},-\xi_{(1)}}([-s^{\prime},0])} \tilde{c}_{(1)}\\
\label{eqn: principal symbol at y: II}  & = \l(r^{-2}\kappa_{(1)}\r)^{\mu+1/2} e^{-\rho_{(1)}(-s^{\prime})} \tilde{\alpha}_{(1)} \mathbf{P}^{A}_{\gamma_{x_{(1)},\xi_{(1)}}([0,s^{\prime}])} \tilde{c}_{(1)},
\end{align}
where the second equality follows from the reversal property of parallel transport maps given in \eqref{eqn: minus parameter parallel transport}.

For convenience, denote
\begin{equation}\label{eqn: notation for calculating the principal symbol}
 \begin{aligned}
  \tilde{c}_{(j)}^{\into} = \mathbf{P}^{A}_{\gamma_{x,\xi}([0,s^{\prime}])}&\tilde{c}_{(j)}\in \C^{n}\b0,\quad \tilde{\alpha}_{(1)}^{\prime} = \l(r^{-2}\kappa_{(1)}\r)^{\mu+1/2}e^{-\rho_{(1)}(- s^{\prime})}\tilde{\alpha}_{(1)}\in \R^{+},\\
  &\alpha_{(k)}^{\prime} = \l(r^{-2}\kappa_{(k)}\r)^{\mu+1/2}e^{-\rho_{(k)}(- s^{\prime})}\tilde{\alpha}_{(k)}\in \R^{+},\quad k=2,3.
  \end{aligned}
\end{equation}

 Moreover, by \eqref{eqn: half-density for transversal intersection}, the half density on $N^{\ast}(\cap_{j=1}^{3}K_{(j)})$ evaluated at $(y,\eta)$ is given by
\begin{equation}\label{eqn: half-density of the interaction source}
   \omega(y,\eta) = \prod_{j=1}^{3} \l(r^{-2}\kappa_{(j)}\r)^{1/2}\omega_{(j)}\l(y,\eta_{(j)}\r).
\end{equation}

Combining \eqref{eqn: principal symbol of the interaction source}, \eqref{eqn: principal symbol at y: I}, \eqref{eqn: principal symbol at y: II}, \eqref{eqn: notation for calculating the principal symbol} and \eqref{eqn: half-density of the interaction source}, the principal symbol of $\chi f_{(123)}$ at $(y,\eta)$ is 
$$
\sigma\l[\chi f_{(123)}\r](y,\eta) = C\omega(y,\eta)|g(y)|^{-1/2} \tilde{\alpha}_{(1)}^{\prime}\tilde{\alpha}^{\prime}_{(2)}\tilde{\alpha}^{\prime}_{(3)} \sum_{\tau\in S(3)} \mathrm{Re}\l\langle  \tilde{c}_{(\tau(1))}^{\into},\tilde{c}_{(\tau(2))}^{\into} \r\rangle c_{(\tau(3))}^{\into}, 
$$
where $C$ is a positive constant only dependent on the codimension of $K_{(1)}$ and $K_{(2)}\cap K_{(3)}$ by Proposition \ref{prop: Greenlead-Uhlmann's calculus}.

Define the intersecting Lagrangian pair as
$$
\Lambda_{0} = N^{\ast}\l(K_{(1)}\cap K_{(2)}\cap K_{(3)} \r),\quad \Lambda_1 \text{ is the future flowout of $\Lambda_0$.}
$$

Since the source $\chi f_{(123)}$ of the third order linearized equation \eqref{Microlocal cut-off equation} is a conormal distribution associated with $\Lambda_0$, as given by \eqref{eqn: microlocal structure of the interaction source}, Theorem \ref{thm: paramatrix of the wave equations} implies that the solution $v_{(123)}$ is a paired Lagrangian distribution:
$$
v_{(123)} \in I^{3\mu+1/2}(M;\Lambda_0,\Lambda_1;E\otimes\Omega^{1/2}).
$$

The principal symbol of $v_{(123)}$ on the bicharacteristic $\beta_{\out}([0,+\infty))\subseteq\Lambda_1$, as given in \eqref{eqn: bicharacteristic}, satisfies the transport equation \eqref{eqn: transport equation for principal symbol} with initial value
\begin{equation}\label{eqn: the principal symbol of v_123 at (y,eta)}
 \sigma\l[v_{(123)}\r](y,\eta) = C\alpha_{\out}\tilde{\alpha}_{(1)}^{\prime}\tilde{\alpha}_{(2)}^{\prime}\tilde{\alpha}_{(3)}^{\prime}\sum_{\tau\in S(3)} \mathrm{Re}\l\langle   \tilde{c}_{\tau(1)}^{\into},\tilde{c}_{\tau(2)}^{\into} \r\rangle \tilde{c}_{\tau(3)}^{\into} ,
\end{equation}
where, by the property of the symbol transition map $\mathscr{R}$ in \eqref{eqn: mapping property of R},
$$
 \alpha_{\out}=|g(y)|^{-1/2}\mathscr{R}\l( \sigma[\Box_g]^{-1}\omega\r)(y,\eta)\in \C\b 0.
$$

We trivialize the half-density bundle on $\Lambda_1$ as follows. By \cite[Theorem 3.1.3]{H}, for all $\lambda_0\in \Lambda_1\b\Lambda_0$, there exists a conic neighborhood $V$ such that 
$$
V\cap \Lambda_1 \subseteq \{(H^{\prime}(\xi),\xi),\xi\in\R^{4}\},
$$ 
where $H$ is a smooth function homogeneous of degree $1$.

Let $\pi$ be the projection from $T^{\ast}M$ to $M$. Then, $v_{(123)}$ near $\pi\lambda_0$ locally takes the form
\begin{equation}\label{eqn: local expression of v_123}
  v_{(123)}(x) = \int_{\R^{4}}e^{\imath (\langle x,\xi\rangle -H(\xi))}b(x,\xi)d\xi,\quad b\in S^{m-1}(\R^{4}\times\R^{4};E\otimes\Omega^{1/2}).
\end{equation}
By \eqref{eqn: density} and \eqref{eqn: principal symbol}, the half-density bundle over $\Lambda_1$ can be locally trivialized by the unity density of order $1/2$ on the fiber of $T^{\ast}M$, that is
\begin{equation}\label{eqn: trivialization of the half-density bundle on Lambda_1}
  |d\xi|^{1/2}\in S^{2}(\Lambda_1;\Omega_{1/2}),
\end{equation}
which is smooth up to $\partial\Lambda_0$.

Next, we locally trivialize the Keller-Maslov bundle over $\Lambda_1$ near the bicharacteristic segment $\beta_{\out}([0,s^{\prime\prime}])$ from $(y,\eta)$ to $(z,\zeta)$. Choose a finite open covering 
$$
\beta_{\out}([0,s^{\prime\prime}]) \subseteq \bigcup_{j=0}^{l} V_{j}, 
$$
where $V_j$ are open conic neighborhoods in $T^{\ast}M\b 0$ such that
\begin{itemize}
  \item $(y,\eta)\in V_1$ and $(z,\zeta)\in V_n$.
  \item $V_j$ only intersects $V_{j-1}$ and $V_{j+1}$, where $V_{-1} = V_{l+1}= \varnothing$ and $j=0,1,\dots,l$.
  \item There exists a partition of $[0,s^{\prime\prime}]$ expressed by $0= s_{0}<s_1<\cdots<s_{l}:=s^{\prime\prime}$ such that $\beta_{\out}([s_{j-1},s_j])\subseteq V_{j}$ for $j=0,1,\dots,l$.
\end{itemize}
The transition functions on the Keller-Maslov bunlde over $\Lambda_1$ are constants, which are powers of $\imath$. Then, we can locally trivialize the Keller-Maslov bundle on $\Lambda_1\cap V_j$ by $\imath^{m_j}$ for some $m_j\in\Z$. Combining this fact and \eqref{eqn: local expression of v_123} and \eqref{eqn: trivialization of the half-density bundle on Lambda_1}, we have
\begin{equation}\label{eqn: principal symbol of v_123 in a small conic neighborhood}
  \sigma\l[v_{(123)}\r]|_{V_j\cap \Lambda_1} = \imath^{m_j}(a_j |d\xi|^{1/2})|_{V_j \cap \Lambda_1},\quad a_{j}\in S^{3\mu-1/2}(\Lambda_1;E).
\end{equation}
Then, we solve the transport equation \eqref{eqn: transport equation for principal symbol} for $\sigma\l[v_{(123)}\r]$ on $\beta_{\out}([0,s^{\prime\prime}])$ by patching the segments $\beta_{\out}([s_{j-1},s_{j}])$. More precisely, write the half-density
$$
  \omega_{j} = |d\xi|^{1/2},\quad\text{on } \Lambda_1 \cap V_j.
$$
Following \cite[Section 2.6]{CLOP}, we have
$$
 \imath^{-m_j}\omega_{j}\mathscr{L}_{H_P}\sigma\l[v_{(123)}\r] =H_{P}a_j+a_j \omega_{j}^{-1} \mathscr{L}_{H_P}\omega_{j}.
$$
By restricting the equation on $\beta_{\out}(s)$ for $s\in (s_{j-1},s_j)
$, we have
$$
 \mathscr{L}_{H_P}\l( \sigma\l[v_{(123)}\r]\r)\circ\beta_{\out}(s) = e^{-\rho_{j}(s)}\partial_{s}\l(e^{\rho_{j}(s)}a_{j}(\beta_{\out}(s)) \r),\quad s\in (s_{j-1},s_j).
$$
where
$$
 \rho_{j}(s) = \int_{s_{j-1}}^{s} \l(\omega_{j}^{-1} \mathscr{L}_{H_P}\omega_{j}\r)(\beta_{\out}(s))ds,\quad s\in (s_{j-1},s_j).
$$
Then, the transport equation \eqref{eqn: transport equation for principal symbol} on $\beta_{\out}(s)$ for $s\in (s_{j-1},s_j)$ with $1\le j\le n$ reduces to a family of equations
$$
\partial_{s}\l(e^{\rho_{j}(s)}a_{j}(\beta_{\out}(s)) \r)+\langle A(\gamma_{y,w}(s)),\dot{\gamma}_{y,w}(s)\rangle = 0,\quad s\in (s_{j-1},s_j),
$$
with initial value $a_{j}(\beta_{\out}(s_{j-1}))$. By \cite[(36)]{CLOP} and the definition of the parallel transport map in \eqref{eqn: parallel transport map in a general parameter}, the equation is solved as
\begin{multline}\label{eqn: solution of the transport equation of amplitude}
 \imath^{-m_j} e^{\rho_{j}(s_j)}\l(\omega_{j}^{-1}\sigma\l[v_{(123)}\r] \r)(\beta_{\out}(s_j)) \\
 =\imath^{-m_j}e^{\rho_{j}(s_{j-1})}\mathbf{P}^{A}_{\gamma_{y,w}([s_{j-1},s_{j}])}\l(\omega_{j}^{-1}\sigma\l[v_{(123)}\r] \r)(\beta_{\out}(s_{j-1})).
\end{multline}
Since $\omega_j$ is a positive valued smooth function on $\Lambda_1 \cap V_j$, we can take $\omega_{j}^{-1}$ out of the parallel transport map as follows
\begin{multline*}
\sigma\l[v_{(123)}\r](\beta_{\out}(s_j)) \\
= e^{\rho_{j}(s_{j-1})-\rho_{j}(s_j)}\omega_{j}(\beta_{\out}(s_j))\omega^{-1}(\beta_{\out}(s_{j-1}))\mathbf{P}^{A}_{\gamma_{y,w}([s_{j-1},s_{j}])}\sigma\l[v_{(123)}\r] (\beta_{\out}(s_{j-1})).
\end{multline*}
Denote the nonvanishing complex valued factor for $j=1,2,\dots,n$,
$$
 C_{j} = e^{\rho_{j}(s_{j-1})-\rho_{j}(s_j)}\omega_{j}(\beta_{\out}(s_j))\omega_{j}^{-1}(\beta_{\out}(s_{j-1}))\in \C\b\{0\}.
$$
Recall the bicharacteristics $\beta_{\out}$ in \eqref{eqn: bicharacteristic}. By patching the solution \eqref{eqn: solution of the transport equation of amplitude} along the partition $\beta_{\out}([s_{j-1},s_j])$ for $j=1,2,\dots,n$, together with the property of the patching of the parallel transport map in \eqref{eqn: patching property of the parallel transport map}, we have
\begin{align*}
\sigma\l[v_{(123)}\r](z,\zeta) &= C_1C_2\cdots C_n\mathbf{P}^{A}_{\gamma_{y,w}([0,s_1])}\mathbf{P}^{A}_{\gamma_{y,w}([s_1,s_2])}\cdots\mathbf{P}^{A}_{\gamma_{y,w}([s_{n-1},s^{\prime\prime}])}\sigma\l[v_{(123)}\r](y,\eta)\\
&=C_1C_2\cdots C_n \mathbf{P}^{A}_{\gamma_{y,w}([0,s^{\prime\prime}])}\sigma\l[v_{(123)}\r](y,\eta).
\end{align*}

Combining it with the expression $\sigma[v_{(123)}]$ at $(y,\eta)$ in \eqref{eqn: the principal symbol of v_123 at (y,eta)}, we have the principal symbol of $v_{(123)}$ at $(z,\zeta)$ as follows.
$$
 \sigma\l[v_{(123)}\r](z,\zeta) = C^{\prime}\tilde{\omega}_{\out}(z,\zeta)\tilde{\alpha}_{(1)}\tilde{\alpha}_{(2)}\tilde{\alpha}_{(3)}\alpha_{\out} \sum_{\tau\in S(3)} \mathrm{Re}\l\langle  \tilde{c}_{(\tau(1))}^{\into},\tilde{c}_{(\tau(2))}^{\into} \r\rangle \mathbf{P}^{A}_{\gamma_{y,w}([0,s^{\prime\prime}])}
 \tilde{c}_{(\tau(3))}^{\into},
$$
where $C^{\prime}\in \C\b \{0\}$ is a constant.

Recall the definition of $\tilde{c}_{(j)}^{\into}$, in which the constant vector $\tilde{c}_{(j)}$ satisfies the identity \eqref{eqn: equality of tilde c_j by construction}. Set $ \tilde{c} := \tilde{c}_{(j)}$ for $j=1,2,3$. Taking $r\to0+$ in \eqref{eqn: perturbation of the vector}, we have the convergence in $L^{+}M$ as follows
$$
\l(x_{(k)},\xi_{(k)} \r)\to \l(x_{(1)},\xi_{(1)} \r):= (x,\xi),\quad k=2,3.
$$
Since the parallel transport maps depend smoothly on the initial points and vectors of the geodesics, it follows that
\begin{align}
\notag \lim_{r\to0+}\sigma\l[v_{(123)}\r](z,\zeta) &= 3 \tilde{\omega}_{\out}(z,\zeta)\tilde{\alpha}_{(1)}\tilde{\alpha}_{(2)}\tilde{\alpha}_{(3)}\alpha_{\out}  \l|\mathbf{P}^{A}_{\gamma_{x,\xi}([0,s^{\prime}])}\tilde{c} \r|^{2} \mathbf{P}^{A}_{\gamma_{x,\xi}([0,s^{\prime}])}\mathbf{P}^{A}_{\gamma_{y,w}([0,s^{\prime\prime}])}\tilde{c}\\
\label{eqn: principal symbol of v_123 after r tends to 0} &=3 \tilde{\omega}_{\out}(z,\zeta)\tilde{\alpha}_{(1)}\tilde{\alpha}_{(2)}\tilde{\alpha}_{(3)}\alpha_{\out}   \mathbf{P}^{A}_{\gamma_{x,\xi}([0,s^{\prime}])}\mathbf{P}^{A}_{\gamma_{y,w}([0,s^{\prime\prime}])}\tilde{c}.
\end{align}
The last equality follows from the fact that parallel transport is $U(n)$-valued, together with the fact that the module of $\tilde{c}$ is $1$ from \eqref{eqn: equality of tilde c_j by construction} by construction of the source.

Recall from \eqref{eqn: regularity of w_123} that $(z,\zeta)\notin \WF(w_{(123)})$. Choose the pseudodifferential operator $\tilde{\chi}$ such that
$$
 \sigma[\tilde{\chi}](z,\zeta) =1 ,\quad \WF(\tilde{\chi}) \cap \WF\l(w_{(123)}\r) = \varnothing.
$$
By the decomposition of $u_{(123)}$ in \eqref{eqn: decompose u_123}, we conclude that
\begin{equation}\label{eqn: microlocal structure of u_123}
 \tilde{\chi}u_{(123)} = \tilde{\chi}v_{(123)} \mod C^{\infty}  \quad\text{and}\quad \tilde{\chi} u_{(123)} \in I^{3\mu+1/2}(\rotom;\Lambda_1,E\otimes\Omega^{1/2}).
\end{equation}
Combining \eqref{eqn: principal symbol of v_123 after r tends to 0} and \eqref{eqn: microlocal structure of u_123}, we obtain that
\begin{multline*}
 \lim_{r\to 0+}\sigma\l[\partial_{\e_{(1)}} \partial_{\e_{(2)}}\partial_{\e_{(3)}}L_{A}\l(\e_{(1)}f_{(1)}+\e_{(2)}f_{(2)}+\e_{(3)}f_{(3)} \r)\r](z,\zeta) \\
 = 3 \tilde{\omega}_{\out}(z,\zeta)\tilde{\alpha}_{(1)}\tilde{\alpha}_{(2)}\tilde{\alpha}_{(3)}\alpha_{\out}   \mathbf{P}^{A}_{\gamma_{x,\xi}([0,s^{\prime}])}\mathbf{P}^{A}_{\gamma_{y,w}([0,s^{\prime\prime}])}\tilde{c}. 
\end{multline*}

Since the factor $\tilde{\omega}_{\out}(z,\zeta)\tilde{\alpha}_{(1)}\tilde{\alpha}_{(2)}\tilde{\alpha}_{(3)}\alpha_{\out}$ is a nonvanishing complex number independent of the connections $A$ and $B$, the coincidence of the source-to-solution maps $L_{A} = L_{B}$ implies that
$$
  \mathbf{P}^{A}_{\gamma_{x,\xi}([0,s^{\prime}])}\mathbf{P}^{A}_{\gamma_{y,w}([0,s^{\prime\prime}])}\tilde{c} =  \mathbf{P}^{B}_{\gamma_{x,\xi}([0,s^{\prime}])}\mathbf{P}^{B}_{\gamma_{y,w}([0,s^{\prime\prime}])}\tilde{c}
$$
Since the above holds for all $\tilde{c}\in \C^{n}$ with $|\tilde{c}| =1$, we conclude that
 $$
  S^{A}_{\gamma_{x,\xi}([0,s^{\prime}])\gamma_{y,w}([0,s^{\prime\prime}])}=S^{B}_{\gamma_{x,\xi}([0,s^{\prime}])\gamma_{y,w}([0,s^{\prime\prime}])}
  $$  
  and the theorem follows.

\end{proof}

\begin{proof}[Proof of Theorem \ref{thm: main theorem}]
 The result follows by combining the recovery of the broken light-ray transform from the source-to-solution map in Theorem \ref{thm: from source-to-solution map to broken light-ray transforms} with the inversion of the transform established in Theorem \ref{thm: inversion of the light-ray transform}. 
\end{proof}

\bibliographystyle{abbrv}
\bibliography{reference}

\end{document}